\documentclass[11pt,twoside, leqno]{article}

\usepackage{mathrsfs}
\usepackage{amssymb}
\usepackage{amsmath}
\usepackage{amsthm}
\usepackage{amsfonts}
\usepackage{extarrows}
\usepackage{xy}
\usepackage{array}
\xyoption{all}

\usepackage{color}
\usepackage[colorlinks, linkcolor=blue, anchorcolor=green, citecolor=blue]{hyperref}
\usepackage[pagewise]{lineno}

\allowdisplaybreaks

\renewcommand\hat{\widehat}
\renewcommand\tilde{\widetilde}

\def\bint{{\ifinner\rlap{\bf\kern.30em--}
\int\else\rlap{\bf\kern.35em--}\int\fi}\ignorespaces}

\def\sbint{{\ifinner\rlap{\bf\kern.32em--}
\hspace{0.078cm}\int\else\rlap{\bf\kern.45em--}\int\fi}\ignorespaces}

\newtheorem{theorem}{Theorem}[section]
\newtheorem{lemma}[theorem]{Lemma}

\theoremstyle{definition}

\newtheorem{definition}[theorem]{Definition}

\numberwithin{equation}{section}

\numberwithin{equation}{section}

\numberwithin{equation}{section}

\begin{document}

\arraycolsep=1pt
\title{\Large\bf Properties of Besov-Lorentz spaces and application to Navier-Stokes equations
\footnotetext{\hspace{-0.15cm}
\endgraf $^\#$\,Corresponding author
}}

\author{  Qixiang Yang$^{1}$, Hongwei Li$^{1,\#}$ }

\date{  }
\maketitle

\vspace{-0.8cm}

\begin{center}
\begin{minipage}{13cm}\small
\begin{center}
\emph{ $^1$\,School of Mathematics and Statistics, Wuhan University, Wuhan }430072\emph{, China}
\end{center}

\begin{center}
\emph{ Email: qxyang@whu.edu.cn., }3368554687\emph{@qq.com}
\end{center}

\hrule
\vspace{\baselineskip}

{\noindent{\bf Abstract}\quad 
Inspired by Caffarelli-Kohn-Nirenberg, Fefferman  and Lin,
we try to investigate how to control the set of large value points for the strong solution of Navier-Stokes equations.
Besov-Lorentz spaces have multiple indices
which can reflect complex changes of the set of the large value points.
Hence we consider some properties of Gauss flow, paraproduct flow and couple flow related to the Besov-Lorentz spaces.
When dealing with Lorentz index,
we use wavelets and maximum norm to describe the decay situation
in the binary time ring and to define time-frequency microlocal maximum norm space.
We use maximum operator, $\alpha$-triangle inequality and H\"older inequality etc to get accurate estimates.
As an application,
we get a global wellposedness result of the Navier-Stokes equations
where the solution can reflect how the size of the set of large value points changes.

}
\vspace{\baselineskip}
\vspace{\baselineskip}

\vspace{\baselineskip}
\hrule
\end{minipage}
\end{center}


\section{Motivations and main theorems}\label{intro}
\par In this paper, we are concerned with some properties of Besov-Lorentz spaces.
Which gives a mathematical description of the change in the set of large value points of a function.
How the set of large value points of a function changes has particular importance
in the study of partial differential equation.
When Fefferman \cite{Fe} introduced the millennium problem on the Navier-Stokes equations,
he specially cites the results of Caffarelli-Kohn-Nirenberg \cite{CKN} and Lin \cite{Lin} on weak solutions.
He said that these two results showed that the measure of the set of the possible blow-up points is zero.
Therefore, we are particularly concerned with how to control the set of large value points of strong solutions.
Fortunately, we can apply the properties of Besov-Lorentz spaces to establish the global wellposedness of the Navier-Stokes equations.

In the fifties and sixties when function space developed greatly,
the space reflecting the distribution of large value points has special meaning.
Many operator continuities are established using Lorentz spaces.
The appearance of Besov-Lorentz space $\dot{B} ^{\frac{n}{p}-1,q} _{p,r}$ is closely related to the interpolation space of Besov space.
See Peetre \cite{P}. For harmonic analysis and differential equations,
the properties of function flows have special implications.
The solutions of many nonlinear equations are found based on the properties of various flows.
The structure of the function flow related to Besov-Lorentz space
is very complicated due to the Lorentz index
which reflects the size of the set of large value points.
Recently Yang-Yang-Hon \cite{YYH} considered some Besov-Lorentz space
$\dot{B} ^{\frac{n}{p}-1,q} _{p,\infty}$ with $q=\infty$.
Unfortunately, their method of generalized Calder\'on-Zygmund decomposition
cannot be applied to the case where $1\leq q<\infty$.
In this paper, we develop a new set of techniques to free up the restrictions on $q$.
We get some properties on three kinds of function flows
related to Besov-Lorentz spaces $\dot{B} ^{\frac{n}{p}-1,q} _{p,\infty}$.

Our first concern is the properties of Gauss flow.
Gauss flow is the solution of parabolic equations
and has been used in the study of many nonlinear problems.
The first step to establish the wellposedness of Navier-Stokes equations
is the boundedness of Gaussian flow.
See Cannone-Meyer-Planchon \cite{CMP}, Kato \cite{Kat}, Kato-Fujita \cite{KF}, Koch-Tataru \cite{KT},
Miura \cite{yang18}, Li-Xiao-Yang \cite{LXY}, Wu \cite{W2} and Yang-Yang \cite{YY}.
We introduce a class of parametric flow space
${ ^{m'}_{m} \dot{B}}^{\frac{n}{p}-1, q}_{p,\infty}$ in Definition \ref{de:3.2}.
${ ^{m'}_{m} \dot{B}}^{\frac{n}{p}-1, q}_{p,\infty}$
is composed of a single norm with nonlinear attenuation depending on frequency and time.
Here $m$ reflects the rapid decay of the function stream at high frequencies.
$m'$ is related to the low frequency and the different stability of trace functions.
${ ^{0}_{m} \dot{B}}^{\frac{n}{p}-1, q}_{p,\infty}\subset L^{\infty} ({\dot{B}}^{\frac{n}{p}-1, q}_{p,\infty})$.
But if $m'>0$, then ${ ^{m'}_{m} \dot{B}}^{\frac{n}{p}-1, q}_{p,\infty}$
is not a subspace of $L^{\infty} ({\dot{B}}^{\frac{n}{p}-1, q}_{p,\infty})$.
For $f(t,x)\in { ^{m'}_{m} \dot{B}}^{\frac{n}{p}-1, q}_{p,\infty}$,
its trace $f(0,x)$ has different sense.
Further, we equip the parametric flow space
${ ^{m'}_{m} \dot{B}}^{\frac{n}{p}-1, q}_{p,\infty}$ with the norm which is completely discreted for the time and frequency.
That is, we construct the norm by maximum estimation in each binary ring of the time.
Our first result shows the boundedness of Gauss flows in Besov-Lorentz spaces.
\begin{theorem} \label{th:B-to-Y}
Given  $1<p< \infty$, $1\leq q< \infty, m'\geq 0,m>0$ or $1<p< \infty$, $q= \infty, m',m>0$.
If $f\in \dot{B} ^{\frac{n}{p}-1,q} _{p,\infty}$, then $ e^{t\Delta} f \in { ^{m'}_{m} \dot{B}}^{\frac{n}{p}-1, q}_{p,\infty}.$
\end{theorem}

Let $P_{j}u$ and $Q^{\epsilon}_{j}v$ be the quantities defined below Lemma \ref{lem:2.1111}.
The sum $\sum_{j}  P_{j-2}u Q^{\epsilon}_{j}v$ is a quantity relative to the paraproduct.
In the 1980s J. M. Bony applied para-product to deal with nonlinear quantities.
At the same time,
R. Coifman and Y. Meyer \cite{Me} also used the concept of para-product
when dealing with the Calder\'on-Zygmund operator.
The compensatory compactness theory of P.L. Lions tells us that
the regularity of a para-product function $f(x)$ can be raised to the Hardy space $H^{1}$
if $f(x)$ belongs to $L^{1}$.
For $l,l',l''\in \{1,\cdots,n\}$,
denote $A_{l} = e^{(t-s)\Delta} \partial_{x_l}$
and $A_{l,l',l''} = e^{(t-s)\Delta} \partial_{x_l}\partial_{x_{l'}}\partial_{x_{l''}}(-\Delta)^{-1}$.
Then we call
$$ P_{l}(u,v)= \!\int^{t}_{0} A_{l} \! \left(\sum\limits_{j}  P_{j-2}u Q^{\epsilon}_{j}v\right)\! ds {\mbox \rm\, and \, }
P_{l,l',l''}(u,v)=\! \int^{t}_{0} A_{l,l',l''}\! \left(\sum\limits_{j}  P_{j-2}u Q^{\epsilon}_{j}v\right)\! ds$$
to be the para-product flow relative to $u$ and $v$.
The second purpose of this paper is to prove the boundedness of the paraproduct flow in the parametric flow space.
\begin{theorem} \label{para}
Given $l,l',l''\in \{1,\cdots,n\}$, $1<p< \infty, 1\leq q\leq \infty, 0\leq m'< \frac{1}{2}, m>1.$
If $u,v\in { ^{m'}_{m} \dot{B}}^{\frac{n}{p}-1, q}_{p,\infty}$, then
$$ P_{l}(u,v), P_{l,l',l''}(u,v)\in { ^{m'}_{m} \dot{B}}^{\frac{n}{p}-1, q}_{p,\infty}.$$
\end{theorem}

Let $Q_{j}v$ be the quantity defined below Lemma \ref{lem:2.1111}.
We call $\sum\limits_{j} Q_j u Q_{j}v$ the coupling product of $u$ and $v$,
which reflects the coupling of two functions $u$ and $v$.
For $l,l',l''\in \{1,\cdots,n\}$, we call
$$ C_{l}(u,v)= \int^{t}_{0} A_{l}\! \left(\sum\limits_{j} Q_j u Q_{j}v\right)\! ds {\mbox \rm\,  and \,}
C_{l,l',l''}(u,v)= \int^{t}_{0} A_{l,l',l''} \! \left(\sum\limits_{j} Q_j u Q_{j}v\right)\! ds$$
to be the couple product flow relative to $u$ and $v$.
For ordinary functions, coupled products are easy to work with,
so only the para product is needed to deal with.
But for the general distribution,
it is easy to deal with the para product
and difficult to deal with the coupled product.
The third purpose of this paper is to prove the boundedness of the couple product flow in the parametric flow space.

\begin{theorem} \label{cou}
Given $l,l',l''\in \{1,\cdots,n\}$, $1<p< \infty, 1\leq q\leq \infty, 0\leq m'< \frac{1}{2}, m>1.$
If $u,v\in { ^{m'}_{m} \dot{B}}^{\frac{n}{p}-1, q}_{p,\infty}$, then
$$ C_{l}(u,v), C_{l,l',l''}(u,v) \in { ^{m'}_{m} \dot{B}}^{\frac{n}{p}-1, q}_{p,\infty}.$$
\end{theorem}

The forth purpose and the main purpose of this paper is to apply the properties of Besov-Lorentz space
to the study of the global wellposedness of the following Navier-Stokes equations \eqref{eq:dirichlet}.
The Cauchy problem of Navier-Stokes equations
on the half-space $\mathbb{R}^{1+n}_+=(0,\infty)\times\mathbb{R}^{n}$ is given as follows:
\begin{equation}\label{eq:dirichlet}
\begin{cases}
\partial_tu-\Delta u+(u\cdot\nabla) u-\nabla p=0,  & in\;\mathbb{R}^{1+n}_+\,;\\
\nabla\cdot u=0,  & in\;\mathbb{R}^{1+n}_+\,;\\
u\big|_{t=0}=f,   & in\;\mathbb{R}^{n}\,.
\end{cases}
\end{equation}
where $u(t,x)$ stands for velocity vector and $p(t,x)$ stands for the pressure of the fluid at the point $(t,x)\in\mathbb{R}^{1+n}_+$.
For $i=1,2,...,n$, let $R_i$ be the Riesz transforms and write
\begin{equation}
\begin{split}
\mathbb{P}=\{\delta_{l,l^{\prime}}+R&_{l}R_{l^{\prime}}\}, l,l^{\prime}=1,2,...,n\,;\\
\mathbb{P} \nabla(u\otimes u)= \sum_{l} \partial_{l}&(u_{l}u)+ \sum_{l} \sum_{l^{\prime}} R_{l} R_{l^{\prime}}\nabla(u_{l} u_{l^{\prime}})\,;  \\
B(u,u)(t,x)=& \int_{0}^{t} e^{(t-s)\Delta}\mathbb {P}\nabla(u\otimes u)ds\,;\\
e^{\widehat{t\Delta}}f(\xi)=&e^{-t\lvert \xi\rvert^2}\hat{f}(\xi)\,.
\end{split}
\end{equation}
The solution to the above Navier-Stokes equations \eqref{eq:dirichlet} is given by the following integral equation.
\begin{equation}\label{eq:0}
u(t,x)=e^{t\Delta}f(x)-B(u,u)(t,x)\,.
\end{equation}
\par The solution to the above integral equation is usually obtained via the following iterative process.
\begin{equation}
\begin{split}
u^{(0)}(t,x)&=e^{t\Delta}f(x)\,;\\
u^{(\tau+1)}(t,x)=u^{(0)}(t,x)-B(u&^{(\tau)},u^{(\tau)})(t,x),   \forall \tau=0,1,2,\cdots\,.
\end{split}
\end{equation}
\par
For some initial value $f\in X$,
if $u^{(\tau)}(t,x)$ belong to iterative space $Y(X)$ and converges to $u(t,x)$ in the iterative space $Y(X)$,
then $u(t,x)$ is called the mild solution of the equations.
The mild solution was first given by Kato-Fujita \cite{KF} in the 1960s.
We can see that there are a great number of studies on the wellposedness of Navier-Stokes equations.
See Kato \cite{Kat}, Miura \cite{yang18}, Meyer-Coifman \cite{MeCo}, Li-Xiao-Yang \cite{LXY} and Yang-Yang \cite{YY} .
For these work, the iterative space $Y(X)\subset L^\infty(X)$.
Until now, the largest initial value space is the ${\rm BMO}^{-1}$ space
in Koch-Tataru \cite{KT}.
But Yang-Yang-Wu \cite{YYW} proved that
Koch-Tataru's iteration space is not a subspace of  $L^\infty({\rm BMO}^{-1}).$

\par  Lorentz type spaces reflect the distribution of large value points of the function.
Since the blow-up phenomenon of the solution of the equations
is closely related to the evolution of the distribution of large value points,
it is an interesting problem how to establish the global wellposedness in the Lorentz type space.
However, Lorentz index seriously affects the estimation of the global wellposedness.
Because when estimating,
you have to deal with the commutativity
between the integrability of the function with respect to time and the Lorentz index.
Barraza \cite{Ba} and Meyer-Coifman \cite{MeCo}  have considered
the homogeneous Lorentz space $\dot{L}^{n,\infty}(\mathbb{R}^{n})
= \{ f(x)\in L^{n,\infty}(\mathbb{R}^{n}), \forall \lambda>0, \lambda f(\lambda x)= f(x)\}
=: L^{n}(S^{n-1})$.
They use the condition that
$f(x) \in \dot{L}^{n,\infty}(\mathbb{R}^{n})\Leftrightarrow f(x)|_{S^{n-1}} \in L^{n}$
to avoid the Lorentz index,
thus obtaining the global wellposedness of the equations \eqref{eq:dirichlet} with initial values in Lorentz space.
Afterwards,
Hobus-Saal \cite{HS} considered the initial data in  Triebel-Lizorkin-Lorentz spaces.
Because local wellposedness is less sensitive to Lorentz index than global wellposedness
when making estimates,
they obtain local wellposedness results for such initial value spaces.
Finding the proper way to control the influence of Lorentz index
is the key to how to get the global wellposedness of the solution
for the initial value in the Lorentz type space.
Yang-Yang-Hon \cite{YYH} considered the Besov-Lorentz spaces $\dot{B}^{\frac{n}{p}-1,q}_{p, \infty}$ with $q=\infty$.
For $q=\infty$,
they could use generalized Calder\'on-Zygmund decomposition and other techniques
to establish the global wellposedness of the solution for the initial value in such a space.
Unfortunately, their method does not apply to the case where $1\leq q<\infty$.
In this paper, we consider the Besov-Lorentz spaces $\dot{B}^{\frac{n}{p}-1,q}_{p, \infty}(1\leq q \leq \infty)$.
We introduce iterative space $ {^{m'}_{m} \dot{B}}^{\frac{n}{p}-1,q}_{p,\infty} $  in Definition \ref{de:3.2}
similar to Yang-Yang-Hon \cite{YYH}.
But we introduced a new set of techniques to control Lorentz index.
As an application of the properties of the previous Besov-Lorentz spaces,
we get the following global wellposedness result:

\begin{theorem}\label{mthmain} Assume that $1<p<\infty, 1\leq q< \infty, 0\leq m'<\frac{1}{2}, m> 1$ or $1<p<\infty, q= \infty, 0< m'<\frac{1}{2}, m> 1$.
The Navier-Stokes equations have a unique smooth solution in
$({^{m'}_{m} \dot{B}}^{\frac{n}{p}-1,q}_{p,\infty} )^n$ for initial data $f(x)$ with $\nabla \cdot f =0$ and $\|f\|_{(\dot{B}^{\frac{n}{p}-1,q}_{p,\infty})^n}$ small
enough.
\end{theorem}

Our iterative space $ {^{m'}_{m} \dot{B}}^{\frac{n}{p}-1,q}_{p,\infty} $ in this paper
is given by the space defined by a single norm with nonlinear attenuation,
where $m$ reflects the rapid decay rate of the high frequency norm with time.
$m'$ here is related to the low frequency,
which can reflect different sense of stability of solution.
For initial data space $X$, the classic stability needs
that the iteration space $Y(X)\subset L^{\infty}(X)$.
We know $ {^{0}_{m} \dot{B}}^{\frac{n}{p}-1,q}_{p,\infty}\subset  L^{\infty}(\dot{B}^{\frac{n}{p}-1,q}_{p,\infty})$.
But if the iteration space $Y(X)$ is not the subspace of $L^{\infty}(X)$,
then there exists some nonlinearity between the initial space $X$ and the iteration space $Y(X)$,
which afford some topological deformation capability.
We know, if $m'>0$, then $ {^{m'}_{m} \dot{B}}^{\frac{n}{p}-1,q}_{p,\infty} $ is not a subspace of $L^{\infty}(\dot{B}^{\frac{n}{p}-1,q}_{p,\infty})$.

In view of the difficulty of Lorentz index, we apply a series of new technique when consturcting the parametric flow space $ {^{m'}_{m} \dot{B}}^{\frac{n}{p}-1,q}_{p,\infty} $ and estimating the norm of the corresponding operators.
First of all, we take an upper bound of time on the ring.
This provides strong support for the exchange order of our estimates of some indicators.
Secondly, at the frequency level, wavelet coefficients are grouped
according to the magnitude of quantity $1+|k'-2^{j'-j}k|$.
Further, we introduce Hardy-Littlewood maximum operator to control the relative estimation.
Thirdly, we consider the two cases $q\leq p$ and $q>p$ via different skills
and we use $\alpha$-triangle inequality and H\"older inequality interchangeably. Etc.
From this, we obtain three independent Theorems \ref{th:B-to-Y}, \ref{para} and \ref{cou}
about some properties of Besov-Lorentz spaces,
and obtain some continuity results about the Gauss flow, the para-product flow, and the coupled flow on Besov-Lorentz spaces.
According to these three continuity results,
we obtain the boundedness of bilinear operator $B(u,v)$ in Section \ref{sec6} and prove Theorem \ref{mthmain}.

\par The rest of the paper is arranged as follows:
In Section \ref{sec2}, we will introduce some preliminaries:
Meyer wavelets, Besov-Lorentz spaces and some estimations of maximum operators.
In Section \ref{sec3}, we introduce time-frequency microlocal space
and prove that such spaces are image spaces of Gauss mappings of Besov-Lorentz spaces.
In Section \ref{sec4}, we will prove the boundedness of paraproduct flow.
In the section \ref{sec5}, we will prove the boundedness of couple flow.
Finally, in the last Section \ref{sec6},
we prove the boundedness of bilinear operator $B(u,v)$
and  establish the global wellposedness for small initial data in our Besov-Lorentz spaces.



\section{Preliminaries}\label{sec2}

In this section, we introduce some preliminary knowledge
relative to wavelets and some basic inequalities.

\subsection{Meyer wavelets and Besov-Lorentz spaces}
The function space we are discussing is described on the basis of wavelets,
see \cite{li29}, \cite{Me}  and \cite{Yang1}.
We introduce Meyer wavelets first.
For the convenience of the subsequent treatment,
we set $E_n=\{0,1\}^n\backslash\{0\}$ ,$\Gamma_n=\{(\epsilon,k): \epsilon\in E_n,k\in\mathbb{Z}^n\}$
and $\Lambda_n=\{(\epsilon,j,k): \epsilon\in E_n, j\in \mathbb{Z}, k\in\mathbb{Z}^n\}$.
\par Let $\phi^0(\xi)$ be a even function in $C^{\infty}_0([-\frac{4\pi}3,\frac{4\pi}3])$
satisfying that $\phi^0(\xi)=1$ when $\lvert \xi\rvert\leq\frac{2\pi}3$ and
$0\leq\phi(\xi)\leq1$.
We set $\varphi(\xi)=[(\phi^0(\frac \xi2))^2-(\phi^0(\xi))^2]^{\frac12}$ and $\phi^1(\xi)=e^{-\frac{i\xi}2}\varphi(\xi)$. Thus we can get when $\lvert
\xi\rvert\leq\frac{2\pi}3$, $\varphi(\xi)=0$; when $\frac{2\pi}3\leq \xi\leq\frac{4\pi}3$, $\varphi^2(\xi)+\varphi^2(2\pi-\xi)=1$.
For arbitrary $\epsilon=(\epsilon_1,\epsilon_2,...,\epsilon_n)\in \{0,1\}^n$,
let $\phi^{\epsilon}(x)$ be the function satisfying
$\hat\phi^{\epsilon}(\xi)=\Pi_{i=1}^n\phi^{\epsilon_i}(\xi_i)$. For any $k\in\mathbb{Z}^n$, $j\in\mathbb{Z}$, set
$\phi^{\epsilon}_{j,k}(x)=2^{\frac{nj}2}\phi^{\epsilon}(2^jx-k)$. We will use $\{\phi^{\epsilon}_{j,k}(x)\}_{(\epsilon,j,k)\in\Lambda_n}$ to denote Meyer wavelets for the rest
of this paper.

For all $\epsilon\in \{0,1\}^{n}, j\in \mathbb{Z}, k\in \mathbb{Z}^{n}$ and distribution $f$, denote
$f^{\epsilon}_{j,k}=\langle f, \phi^{\epsilon}_{j,k}\rangle$.
Using Meyer wavelets, we can characterize $L^2(\mathbb{R}^n)$ in the following way:
\begin{lemma}\label{lem:2.1111}
The Meyer wavelets make up an orthogonal basis of $L^2(\mathbb{R}^n)$.
Further, for any function $f(x)\in L^2(\mathbb{R}^n)$,
$f(x)=\sum_{(\epsilon,j,k)\in\Lambda_n}f^{\epsilon}_{j,k}\phi^{\epsilon}_{j,k}(x)$ in the $L^2$ convergence sense.
\end{lemma}
We need to use several orthogonal projection operators to consider the product of functions.
For $j\in\mathbb{Z}$ and $\epsilon\in \{0,1\}^{n}\backslash \{0\}$, let
$$\begin{array}{rcl}
P_{j}f(x)= \sum\limits_{k\in \mathbb{Z}^{n}} f^{0}_{j,k}
\phi^{0}_{j,k}(x),&& \\
Q_{j}^{\epsilon}f(x)= \sum\limits_{k \in \mathbb{Z}^n
} f^{\epsilon}_{j,k} \phi^{\epsilon}_{j,k}(x)
&\text{ and }& Q_{j}f(x)= \sum\limits_{\epsilon \in
E_n} Q^{\epsilon}_{j} f(x). \end{array}$$

Lorentz spaces $L^{p,r}$ are the generalization of Lebesgue spaces $L^{p}$.
Barraza \cite{Ba} and Coifman-Meyer \cite{MeCo} have established
the global wellposedness for initial value in homogeous
Lorentz spaces $L^{n,\infty}$ by identifying such function with Lebesgue function on the unit sphere.
Besov-Lorentz spaces $\dot{B}^{s,q}_{p,r}=l^{s,q}(L^{p,r})$, introduced in Peetre \cite{P},
can be regarded as the generalization of Besov spaces.
We dominate the measure of the set of large value points by the norm in this kind of spaces.
Further, some of Besov-Lorentz space can be treated as the interpolation space of Besov space and the space of bounded variation in the sense of Wiener, resp., see \cite{P}.
Triebel \cite{T} and Yang-Cheng-Peng \cite{YCP} introduced Triebel-Lizorkin-Lorentz spaces $F^{s,q}_{p,r}$,
which can be regarded as the generalization of Triebel-Lizorkin spaces and Lorentz spaces.
Hobus-Saal \cite{HS} have established local wellposedness for initial value in $F^{s,q}_{p,r}$.
Let $\chi(x)$ be the characterization function on the unit interval $[0,1]^{n}$ and denote
$f_{j}(x) = 2^{\frac{n}{2}j} \sum\limits_{(\epsilon,k)\in \Gamma_n}|f^{\epsilon}_{j,k}|\chi(2^{j}x-k).$  The forthcoming result is well-known. See \cite{Me, YCP}:
\begin{lemma}\label{le:2.2}

{\rm(i)} Given $s\in \mathbb{R}$ and $1\leq p,q\leq \infty$.
$f(x)\in \dot{B}^{s,q}_{p} = l^{s,q}(L^{p})\Longleftrightarrow$
$$\begin{array}{c}
\sum\limits_{j}  2^{jq(s+\frac{n}{2}-\frac{n}{p})}  \left( \sum\limits_{(\epsilon, k)\in\Gamma_n}\vert f^{\epsilon}_{j,k}\vert^{p}\right)
^{\frac{q}{p}}<\infty.
\end{array}$$

{\rm(ii)} Given $s\in \mathbb{R}$ and $1\leq p,q,r\leq \infty$.
$f(x)\in \dot{B}^{s,q}_{p,r} = l^{s,q}(L^{p,r})\Longleftrightarrow$
$$\begin{array}{c}
\sum\limits_{j}  2^{jqs}  \{ \sum\limits_{u\in \mathbb{Z}} 2^{u r} |\{x: f_{j}(x)>2^{u}\}|^{\frac{r}{p}}\}^{\frac{q}{r}}<\infty.
\end{array}$$
{\rm(iii)} Given $s\in \mathbb{R}$, $1<p,q, r< \infty$.
$f(x)\in \dot{F}^{s,q}_{p,r} = L^{p,r} (l^{s,q})\Longleftrightarrow$
$$\begin{array}{c}
\sum\limits_{u}  2^{ur}  | \{x: \sum\limits_{j} 2^{jsq}|f_{j}(x)|^ q >2^{qu}\}|^{\frac{r}{p}}<\infty.
\end{array}$$
\end{lemma}


\par
Considering that they are closely linked to scale of the set of large value points
and Besov-Lorentz can improve the interpolation of the operator,
the study of the properties on Besov-Lorentz spaces is a matter of general interest.
In this paper, we study the continuity of some operators in Section \ref{sec3}, \ref{sec4} and \ref{sec5}.
Further, we apply these properties to establish the global wellposedness for initial data in Besov-Lorentz spaces.


\subsection{Basic inequalities}

The following two inequalities can be proved easily:
\begin{lemma}\label{le:2.3}
If $a\geq 1$, then
$$1+\vert x\vert \leq 2(1+\vert x-y\vert)(1+a\vert y \vert).$$

\end{lemma}

\begin{proof}
We distinguish two cases: (i) $a\vert y\vert\leq \frac{\vert x\vert}{2}$ and (ii) $a\vert y \vert\geq \frac{\vert x\vert}{2}$.
We get the above conclusion.
\end{proof}
\begin{lemma}\label{le:2.4}
For $0<r\leq1$ and $a_k\geq0,\;k\in N_+$, we have
$$(\sum_{k>0}a_k)^r\leq \sum_{k>0}(a_k)^r  .$$
\end{lemma}
\begin{proof}
For any $k>0$, let $b_k=(a_k)^r$. Then $$\sum_{k>0}a_k=\sum_{k>0}(b_k)^{\frac 1r}\leq\Vert b\Vert_{\infty}^{\frac 1r-1}\sum_{k>0}b_k\leq(\sum_{k>0}b_k)^{\frac
1r}=[\sum_{k>0}(a_k)^r]^{\frac 1r}  .$$
\end{proof}

Further, by the above Lemma \ref{le:2.3}, we have:
\begin{lemma}\label{le:2.5}
For $ j\geq j'$, we have
$$\begin{array}{rl}
&(1+\vert 2^{j'}x-k'\vert)^{-n-1} (1+ \vert 2^{j}x-k\vert)^{-N-n-1} \\
\lesssim &   (1+ \vert 2^{j'-j}k-k'\vert)^{-n-1} (1+ \vert 2^{j} x-k \vert)^{-N}.
\end{array}$$
\end{lemma}

In this paper, we have to use some estimations relative to Hardy-Littlewood maximum operator $M$.
For $\{f^{\epsilon}_{j,k}\}_{(\epsilon,j,k)\in \Lambda_n}$,
let $f_{j}(x) = 2^{\frac{n}{2}j} \sum\limits_{(\epsilon,k)\in \Gamma}|f^{\epsilon}_{j,k}|\chi(2^{j}x-k)$ and let
\begin{equation*}
\begin{split}
g^k_{j,j'}=\left\{
\begin{array}{ll}
\sum_{(\epsilon',k')\in\Gamma}\frac{2^{\frac{n}{2}j'}|f^{\epsilon'}_{j',k'}|}{(1+|k'-2^{j'-j}k|)^N},&j\geq j',k\in\mathbb{Z}^n;\\
\sum_{(\epsilon',k')\in\Gamma}\frac{2^{\frac{n}{2}j'}|f^{\epsilon'}_{j',k'}|}{(1+|k-2^{j-j'}k'|)^N},&j<j',k\in\mathbb{Z}^n.
\end{array}
\right.
\end{split}
\end{equation*}

Yang \cite{Yang1} has proved the following Lemma:
\begin{lemma}\label{le:2.6}
For $N>2n+1$ and $x\in Q_{j,k}$, we have
\begin{equation*}
\begin{split}
g^k_{j,j'}\lesssim\left\{
\begin{array}{ll}
M(f_{j'})(x),&j\geq j'\\2^{n(j'-j)}M(f_{j'})(x),&j<j'
\end{array}
\right.
\end{split}
\end{equation*}
\end{lemma}
\begin{proof}
In fact, we only need to consider the cases where $j\geq j'$ and $j<j'$ respectively. For each given $(j,k)$ and $(j',k')$, we should considering the least dyadic cube
containing $Q_{j,k}$ and $Q_{j',k'}$ for maximum operators. We then obtain the lemma after direct calculation. See \cite{Yang1}, Lemma 3.2, Chapter 5 for details.
\end{proof}

We group the wavelets coefficients according to the the quantity of $1+|k'-2^{j'-j}k|$.
Then get the estimation of spatial variables by the attenuation of the $1+|k'-2^{j'-j}k|$.
We use the maximum operator to dominate the Lorentz index in this manner.

Chung-Hunt-Kurtz \cite{max}
considered the boundedness of Hardy-Littlewood maximum operator $M$
on the Lorentz spaces $L^{p,\infty}$
and the forthcoming lemma can be found in Theorem 3 of Chung-Hunt-Kurtz \cite{max}:
\begin{lemma}\label{le:2.7}
For all $1<p<\infty$, we have
\begin{equation*}
\sup_{\lambda>0}\lambda^p|\{x:Mf(x)>\lambda\}|\lesssim \sup_{\lambda>0}\lambda^p|\{x:f(x)>\lambda\}|  .
\end{equation*}
\end{lemma}


\section{Microlocal maximum norm space and Gauss flow}\label{sec3}
An initial value space $X$ is called critical for the
Navier-Stokes equations (\ref{eq:dirichlet}), if $\Vert u_{\theta}(\cdot)\Vert_X=\Vert u(\cdot)\Vert_X$.
Since Navier-Stokes equations have scale invariance, critical space is of special importance.
If $u(t,x)$ is a solution of (\ref{eq:dirichlet}) with $p(t,x)$ as pressure and $f(x)$ as initial value,
then for $\theta>0$, we replace $u(t,x)$, $p(t,x)$ and $f(x)$ with $u_{\theta}(t,x) = \theta u(\theta^{2}t, \theta x)$, $p_{\theta}(t,x) = \theta^{2} p(\theta^{2}t, \theta x)$ and
$f_{\theta}(x)=\theta\cdot f(\theta x)$ respectively and the Equations (\ref{eq:dirichlet}) still hold.
For the wellposedness of Navier-Stokes equations,
many scholars have studied it in different initial value spaces before.
Kato-Fujita \cite{KF} consider first the mild solution for $H^{s}$.
Kato \cite{Kat} consider the global wellposedness of Navier-Stokes equations for initial value in $\dot{L}^n(\mathbb{R}^n)$.
Li-Xiao-Yang \cite{LXY} and Wu \cite{W2} considered the existence of solutions on some critical Besov spaces and critical Besov-Morrey spaces.
Koch-Tataru \cite{KT} and Miura \cite{yang18} consider the biggest initial data space $\textrm{BMO}^{-1}$.
In this paper, we consider critical Besov-Lorentz spaces.

\par In this section, we introduce the single norm parameter space ${ ^{m'}_{m} \dot{B}}^{\frac{n}{p}-1, q}_{p,\infty}$.
In the sequel, we will establish the relationship between
Besov-Lorentz space
$\dot{B}^{\frac{n}{p}-1,q}_{p,\infty}$
and the parameter space ${ ^{m'}_{m} \dot{B}}^{\frac{n}{p}-1, q}_{p,\infty}$.
That is, we prove Theorem \ref{th:B-to-Y}.


\subsection{Microlocal maximum norm space}\label{sec3.1}

\par For any $f\in\dot{B}^{\frac{n}{p}-1,q}_{p,\infty}$ and $f(t,x)=e^{t\Delta}f$,
for $(\epsilon,j,k)\in\Lambda_n$,
denote  $f^{\epsilon}_{j,k}(t)=\langle f(t,\cdot),\phi^{\epsilon}_{j,k}\rangle$ and
$f^{\epsilon}_{j,k}=\langle f,\phi^{\epsilon}_{j,k}\rangle$.
Hence
$$f=\sum_{(\epsilon,j,k)\in\Lambda_n}f^{\epsilon}_{j,k}\phi^{\epsilon}_{j,k} {\mbox \, and \,} f(t,x)=\sum_{(\epsilon,j,k)\in\Lambda_n}f^{\epsilon}_{j,k}(t)\phi^{\epsilon}_{j,k}(x).$$
According to the properties of the Fourier transform of Meyer wavelet, we have
$$f^{\epsilon}_{j,k}(t)=\sum_{\epsilon^{\prime},\lvert
j-j^{\prime}\rvert\leq1,k^{\prime}}f^{\epsilon^{\prime}}_{j^{\prime},k^{\prime}}\\\langle
e^{t\Delta}\phi^{\epsilon^{\prime}}_{j^{\prime},k^{\prime}},\phi^{\epsilon}_{j,k}\rangle.$$

At the Beijing Conference on Harmonic Analysis and Its Applications in 2023,
Professor Zhifei Zhang specifically mentioned how to improve the iteration space of
Cannone-Meyer-Planchon \cite{CMP}.
In this paper, we introduce microlocal maximum norm space as the iteration space,
which is a special class of nonlinear decaying spaces with a single norm.
For convenience, we introduce the following notations.
For any $j, j_t\in \mathbb{Z}$, denote
\begin{align*}
f_{j}(t,x) = &2^{\frac{n}{2}j} \sum\limits_{(\epsilon,k)\in \Gamma}|f^{\epsilon}_{j,k}(t)|\chi(2^{j}x-k);\\
f_{j,j_t}(x)=&\sup_{2^{-2j_t}\leq t<2^{2-2j_t}}f_{j}(t,x).
\end{align*}
For all $\lambda>0$, $j\in \mathbb{Z},m\in\mathbb{R}$,  denote

$$\begin{array}{c}
A^{p, m}_{j,j_t} = 2^{2(j-j_t)m} 2^{j(\frac np-1)}
\sup_{\lambda>0} \lambda |\{x: f_{j,j_t}( x)>\lambda\}|^{\frac{1}{p}}<\infty\\

\end{array}$$
and for $q<\infty$, $A^{p,m}=\sup_{j_t\in \mathbb{Z}}\sum_{j\geq j_t}(A^{p,m}_{j,j_t})^q$, $A^{p,m'}=\sup_{j_t\in \mathbb{Z}}\sum_{j< j_t}(A^{p,m'}_{j,j_t})^q$; for
$q=\infty$, $A^{p,m}=\sup_{j_t\in \mathbb{Z},j\geq j_t}A^{p,m}_{j,j_t}$, $A^{p,m'}=\sup_{j_t\in \mathbb{Z},j< j_t}A^{p,m'}_{j,j_t}$.
In this paper,
we introduce a new type flow space ${ ^{m'}_{m} \dot{B}}^{\frac{n}{p}-1, q}_{p,\infty}$
with nonlinear single norm attenuation.

\begin{definition}\label{de:3.2} Given $1\leq p<\infty, 1\leq q\leq \infty, m'\geq 0, m>0.$
$f(t,x)\in { ^{m'}_{m} \dot{B}}^{\frac{n}{p}-1, q}_{p, \infty}$ if and only if
\begin{align*}
\begin{split}
\left\{\begin{array}{ll}
\sup_{j_t\in\mathbb{Z}}(\sum\limits_{j\geq j_t} \{A^{p, m}_{j,j_t}\}^{q}
+ \sum\limits_{j< j_t} \{A^{p, m'}_{j,j_t}\}^{q})<\infty ,&1\leq q<\infty;\\
\sup_{j_t\in\mathbb{Z},j\geq j_t}A^{p, m}_{j,j_t}
+ \sup_{j_t\in\mathbb{Z},j< j_t}A^{p, m'}_{j,j_t}<\infty ,&q=\infty.
\end{array}\right.
\end{split}
\end{align*}
\end{definition}

Different to which in Yang-Yang-Hon \cite{YYH},
our ${ ^{m'}_{m} \dot{B}}^{\frac{n}{p}-1, q}_{p,\infty}$
is a function space of maximum norm over different binary rings for both time and frequency.
For $1\leq q<\infty$,
we can't control the Lorentz index directly in the semi-discrete flow Lorentz space.
But we found that if we take the maximum norm over the time-dependent ring,
we can control the Lorentz index with appropriate techniques.

\subsection{Properties of microlocal space}\label{sec3.2}


Define parameterized Besov space $B_{m,m'}$ as follows:
$$(t2^{2j})^{m} 2^{(\frac{n}{2}-1)j} |f^{\epsilon}_{j,k}(t)|<\infty, \forall t2^{2j}\geq 1.$$
$$(t2^{2j})^{m'} 2^{(\frac{n}{2}-1)j} |f^{\epsilon}_{j,k}(t)|<\infty, \forall 0<t2^{2j}\leq 1.$$
By the definition of ${ ^{m'}_{m} \dot{B}}^{\frac{n}{p}-1, q}_{p,\infty}$,
we have:
\begin{lemma}
Given $1\leq p<\infty, 1\leq q\leq \infty, m'\geq 0, m>0,$
$${ ^{m'}_{m} \dot{B}}^{\frac{n}{p}-1, q}_{p,\infty}\subset B_{m,m'}.$$
\end{lemma}

\begin{proof}
For $E\subset \mathbb{Z}^{n}$, denote $\# E$ the number of elements in $E$.
For any $j_t\in\mathbb{Z}$ and $j\geq j_t$, we have
$$\sum_{j\geq j_t}[2^{2(j-j_t)m}2^{j(\frac np-1)}\sup_{\lambda>0}\lambda\cdot 2^{-\frac{nj}p}\#\{k:\sup_{2^{-2j_t}\leq t<2^{2-2j_t}}\sum_{\epsilon\in
E_n}|f^{\epsilon}_{j,k}(t)|>\lambda\cdot2^{-\frac n2 j}\}^{\frac 1p}]^q<\infty  .$$

Hence$$\sup_{\lambda>0}2^{2(j-j_t)m}2^{-j}\cdot\lambda\#\{k:\sup_{2^{-2j_t}\leq t<2^{2-2j_t}}\sum_{\epsilon\in E_n}|f^{\epsilon}_{j,k}(t)|>\lambda\cdot2^{-\frac n2 j}\}^{\frac
1p}<\infty  .$$
For $q=\infty$ we can directly get the above estimate.
Let $\lambda_{j,j_t}>0$ and it satisfies that
$$\#\{k:\sup_{2^{-2j_t}\leq t<2^{2-2j_t}}\sum\limits_{\epsilon\in E_n}|f^{\epsilon}_{j,k}(t)|>\lambda_{j,j_t}\cdot2^{-\frac n2 j}\}>0$$ and
$$\#\{k:\sup_{2^{-2j_t}\leq t<2^{2-2j_t}}\sum\limits_{\epsilon\in E_n}|f^{\epsilon}_{j,k}(t)|>2\lambda_{j,j_t}\cdot2^{-\frac n2 j}\}=0.$$ We have
$2^{2m(j-j_t)}2^{-j}\cdot\lambda_{j,j_t}<\infty$ and for any $k\in\mathbb{Z}^n$, $\sup_{\epsilon\in E_n,2^{-2j_t}\leq
t<2^{2-2j_t}}|f^{\epsilon}_{j,k}(t)|\leq2\lambda_{j,j_t}\cdot2^{-\frac n2 j}$. According to the arbitrariness of $j_t$, we obtain
$$(t2^{2j})^{m}2^{-j}\cdot2^{\frac n2 j}\sup\limits_{\epsilon\in E_n}|f^{\epsilon}_{j,k}(t)|<\infty, \forall t2^{2j}\geq1.$$
And that means $(t2^{2j})^m2^{j(\frac n2-1)}|f^{\epsilon}_{j,k}(t)|<\infty$.
The same is true for $j<j_t$.

\end{proof}

By definition of ${ ^{m'}_{m} \dot{B}}^{\frac{n}{p}-1, q}_{p,\infty}$ and $B_{m,m'}$,
we can easily get  the following estimation.

\begin{lemma}
Given $1\leq p < \infty, 1\leq q\leq \infty, 0\leq m'<\frac 12, m>1.$
$f\in { ^{m'}_{m} \dot{B}}^{\frac{n}{p}-1, q}_{p, \infty}$ implies
$$\sup\limits_{k\in \mathbb{Z}^n } 2^{\frac{nj}{2}} \left\vert f^{0}_{j,k}(t)\right\vert
\lesssim  t^{-\frac{1}{2}},\; for \;all\; t2^{2j}\geq 1.$$
$$\sup\limits_{k\in \mathbb{Z}^n } 2^{\frac{nj}{2}} \left\vert f^{0}_{j,k}(t)\right\vert
\lesssim 2^{j(1-2m')} t^{-m'},\; for \;all\; 0<t2^{2j}\leq 1.$$
\end{lemma}

\begin{proof}
 Because
$\lvert f^{0}_{j,k}(t)\rvert\lesssim \sum_{j^{\prime}<j-1,\epsilon^{\prime},k^{\prime}}\lvert f^{\epsilon^{\prime}}_{j^{\prime},k^{\prime}}(t)\langle
\phi^{\epsilon^{\prime}}_{j^{\prime},k^{\prime}} ,\phi^0_{j,k}\rangle\rvert$. Furthermore we have ${^{m^{\prime}}_{m}\dot{B}}^{\frac np-1,q}_{p,\infty}\in B_{m,m^{\prime}}$,
which derives that:\par
(1). For $0<t2^{2j}\leq1$,
\begin{equation*}
\begin{split}
\lvert f^{0}_{j,k}(t)\rvert\lesssim&\sum_{j^{\prime}<j-1,\epsilon^{\prime},k^{\prime}} (t2^{2j^{\prime}})^{-m^{\prime}}2^{(1-\frac n2)j^{\prime}}\int
2^{\frac{n(j+j^{\prime})}2}\phi^{\epsilon^{\prime}}(2^{j^{\prime}}x-k^{\prime})\phi^{0}(2^jx-k)dx\\
\lesssim&2^{\frac{nj}2}t^{-m^{\prime}}\sum_{j^{\prime}<j-1}2^{(1-2m^{\prime})j^{\prime}}\sum_{\epsilon^{\prime},k^{\prime}}\int
(1+\lvert2^{j^{\prime}}x-k^{\prime}\rvert)^{-N}(1+\lvert2^{j}x-k\rvert)^{-N}dx\\
\lesssim&2^{\frac{nj}2}t^{-m^{\prime}}\sum_{j^{\prime}<j-1}2^{(1-2m^{\prime})j^{\prime}}\sum_{\epsilon^{\prime},k^{\prime}}(1+\lvert2^{j^{\prime}-j}k-k^{\prime}\rvert)^{-N}\int
(1+\lvert2^{j}x-k\rvert)^{-N}dx\\
\lesssim&2^{\frac{nj}2}t^{-m^{\prime}}\sum_{j^{\prime}<j-1}2^{(1-2m^{\prime})j^{\prime}}\cdot 2^{-nj}\lesssim2^{(1-2m^{\prime}-\frac n2)j}t^{-m^{\prime}}\,.
\end{split}
\end{equation*}
(2). For $t2^{2j}\geq1$,
\begin{equation*}
\begin{split}
\lvert f^{0}_{j,k}(t)\rvert\lesssim&\sum_{j^{\prime}\leq-\frac{\log_2t}2,\epsilon^{\prime},k^{\prime}} (t2^{2j^{\prime}})^{-m^{\prime}}2^{(1-\frac n2)j^{\prime}}\int
2^{\frac{n(j+j^{\prime})}2}\phi^{\epsilon^{\prime}}(2^{j^{\prime}}x-k^{\prime})\phi^{0}(2^jx-k)dx+\\
&\sum_{-\frac{\log_2t}2\leq j^{\prime}\leq j,\epsilon^{\prime},k^{\prime}} (t2^{2j^{\prime}})^{-m}2^{(1-\frac n2)j^{\prime}}\int
2^{\frac{n(j+j^{\prime})}2}\phi^{\epsilon^{\prime}}(2^{j^{\prime}}x-k^{\prime})\phi^{0}(2^jx-k)dx\\
\lesssim &2^{(1-2m^{\prime})\cdot(-\frac{\log_2t}2)}t^{-m^{\prime}}2^{-\frac {nj}2}+2^{(1-2m)\cdot(-\frac{\log_2t}2)}2^{-\frac n2j}t^{-m}\\
\lesssim&t^{-\frac 12}2^{-\frac {nj}2}+t^{m-\frac 12}t^{-m}2^{-\frac {nj}2}\lesssim t^{-\frac 12}2^{-\frac {nj}2}\,.
\end{split}
\end{equation*}
\end{proof}

\subsection{Proof of Theorem \ref{th:B-to-Y}}\label{sec3.3}

The following lemma can be found in \cite{LXY}.
\begin{lemma}\label{le:3.1}
There exists a constant $N^{\prime}\in\mathbb{N}_+$ large enough and a small constant $\widetilde{c}>0$ that for any positive integer $N$, as long as $N>N^{\prime}$, then
\begin{equation}
\lvert f^{\epsilon}_{j,k}(t)\rvert\lesssim e^{-\widetilde{c}t2^{2j}}\sum_{\epsilon^{\prime},\lvert j-j^{\prime}\rvert\leq 1,k^{\prime}}\lvert
f^{\epsilon^{\prime}}_{j^{\prime},k^{\prime}}\rvert(1+\lvert2^{j-j^{\prime}}k^{\prime}-k\rvert)^{-N},\forall t2^{2j}\geq1
\end{equation}
and
\begin{equation}
\lvert f^{\epsilon}_{j,k}(t)\rvert\lesssim \sum_{\epsilon^{\prime},\lvert j-j^{\prime}\rvert\leq 1,k^{\prime}}\lvert
f^{\epsilon^{\prime}}_{j^{\prime},k^{\prime}}\rvert(1+\lvert2^{j-j^{\prime}}k^{\prime}-k\rvert)^{-N},\forall 0< t2^{2j}\leq1\,.
\end{equation}
\end{lemma}

 For $0<t2^{2j}< 1$, by lemma \ref{le:3.1}, we have
$$\lvert f^{\epsilon}_{j,k}(t)\rvert\lesssim\sum_{\epsilon^{\prime},\lvert j-j^{\prime}\rvert\leq1,k^{\prime}}\lvert
f^{\epsilon^{\prime}}_{j^{\prime},k^{\prime}}\rvert(1+\lvert 2^{j-j^{\prime}}k^{\prime}-k\rvert)^{-N}.$$
Using lemma \ref{le:2.6}, finally we get
$$\lvert f_j(t,x)\rvert\lesssim\sum_{\lvert j-j^{\prime}\rvert\leq1} M(f_{j^{\prime}})(x).$$
For any $j_t\in\mathbb{Z}$, let $f_{j,j_t}(x)=\sup_{2^{-2j_t}\leq t<2^{2-2j_t}}f_j(t,x)$ as we denote in Section \ref{sec3.1}.
We obtain
\begin{equation*}
\begin{split}
\lvert \{x:f_{j,j_t}(x)>\lambda\}\rvert\leq &\lvert \{x:\sum_{\lvert j-j^{\prime}\rvert\leq1}CM\lvert f_{j^{\prime}}(x)\rvert>\lambda\}\rvert \\\leq&\sum_{\lvert
j-j^{\prime}\rvert\leq1}\lvert \{x:CM\lvert f_{j^{\prime}}(x)\rvert>\frac{\lambda}3\}\rvert\,.
\end{split}
\end{equation*}
For $1\leq q<\infty$ and $m'\geq0$, let $A^{p,m^{\prime}}=\sup_{j_t\in\mathbb{Z}}\sum_{j<j_t}2^{2(j-j_t)m'q}2^{jq(\frac
np-1)}\sup_{\lambda>0}\lambda^q\lvert\{x:f_{j,j_t}(x)>\lambda\}\rvert^{\frac qp}$. Then
\begin{equation*}
\begin{split}
A^{p,m^{\prime}}\lesssim\sup_{j_t\in\mathbb{Z}}\sum_{j<j_t}2^{2(j-j_t)m'q}2^{jq(\frac np-1)}\sup_{\lambda>0}\lambda^q(\sum_{\lvert j-j^{\prime}\rvert\leq1}\lvert \{x:CM\lvert
f_{j^{\prime}}(x)\rvert>\frac{\lambda}3\}\rvert)^{\frac qp}\,.
\end{split}
\end{equation*}
For $q>p$, we apply H\"older inequality; for $q\leq p$, we consider the following $\alpha$-triangle inequality:
$$ (a+b)^{\alpha}\lesssim a^{\alpha}+b^{\alpha}  \qquad \mbox{for} \;a,b\geq 0,\;0<\alpha\leq1\,.$$
Then we have
\begin{equation*}
\begin{split}
A^{p,m^{\prime}}\lesssim&\sup_{j_t\in\mathbb{Z}}\sum_{j<j_t}2^{2(j-j_t)m'q}2^{jq(\frac np-1)}\sum_{\lvert j-j^{\prime}\rvert\leq1}\sup_{\lambda>0}\lambda^q\lvert \{x:CM\lvert
f_{j^{\prime}}(x)\rvert>\frac{\lambda}3\}\rvert^{\frac qp}\\\lesssim&\sup_{j_t\in\mathbb{Z}}\sum_{j<j_t}2^{2(j-j_t)m'q}2^{jq(\frac np-1)}\sum_{\lvert
j-j^{\prime}\rvert\leq1}\sup_{\lambda>0}\lambda^q\lvert \{x:M\lvert f_{j^{\prime}}(x)\rvert>\frac{\lambda}{3C}\}\rvert^{\frac
qp}\\\lesssim&\sup_{j_t\in\mathbb{Z}}\sum_{j<j_t}2^{2(j-j_t)m'q}2^{jq(\frac np-1)}\sum_{\lvert j-j^{\prime}\rvert\leq1}\Vert M\lvert f_{j'}\rvert\Vert^q_{L^{p,\infty}}\\
\lesssim&\sup_{j_t\in\mathbb{Z}}\sum_{j'<j_t+1}\sum_{j:\lvert j-j^{\prime}\rvert\leq1}2^{2(j'-j_t)m'q}2^{j'q(\frac np-1)}\Vert
f_{j'}\Vert^q_{L^{p,\infty}}\\\lesssim&\sup_{j'\in\mathbb{Z}}2^{j'q(\frac np-1)}\Vert  f_{j'}\Vert^q_{L^{p,\infty}}=\Vert f\Vert_{{\dot{B}}^{\frac{n}{p}-1, q}_{p,\infty}}\,.
\end{split}
\end{equation*}

For $q=\infty$ and $m'>0$, let $A^{p,m^{\prime}}=\sup_{j_t\in\mathbb{Z},j,j_t}2^{2(j-j_t)m'}2^{j(\frac np-1)}\sup_{\lambda>0}\lambda\lvert\{x:f_{j,j_t}(x)>\lambda\}\rvert^{\frac 1p}$.
Hence
\begin{equation*}
\begin{split}
A^{p,m^{\prime}}\lesssim&\sup_{j_t\in\mathbb{Z},j<j_t}2^{2(j-j_t)m'}2^{j(\frac np-1)}\sum_{\lvert j-j^{\prime}\rvert\leq1}\sup_{\lambda>0}\lambda\lvert \{x:CM\lvert
f_{j^{\prime}}(x)\rvert>\frac{\lambda}3\}\rvert^{\frac 1p}\\\lesssim&\sup_{j_t\in\mathbb{Z},j<j_t}2^{2(j-j_t)m'}2^{j(\frac np-1)}\sum_{\lvert
j-j^{\prime}\rvert\leq1}\sup_{\lambda>0}\lambda\lvert \{x:M\lvert f_{j^{\prime}}(x)\rvert>\frac{\lambda}{3C}\}\rvert^{\frac
1p}\\\lesssim&\sup_{j_t\in\mathbb{Z},j<j_t}2^{2(j-j_t)m'}2^{j(\frac np-1)}\sum_{\lvert j-j^{\prime}\rvert\leq1}\Vert M\lvert f_{j'}\rvert\Vert_{L^{p,\infty}}\\
\lesssim&\sup_{j_t\in\mathbb{Z}}\sup_{j<j_t}\sum_{\lvert j-j^{\prime}\rvert\leq1}2^{2(j'-j_t)m'}2^{j'(\frac np-1)}\Vert
f_{j'}\Vert_{L^{p,\infty}}\\\lesssim&\sup_{j_t\in\mathbb{Z}}\sum_{j'<j_t+1}\sup_{j:\lvert j-j^{\prime}\rvert\leq1}2^{2(j'-j_t)m'}2^{j'(\frac np-1)}\Vert
f_{j'}\Vert_{L^{p,\infty}}\lesssim\sup_{j'\in\mathbb{Z}}2^{j'(\frac np-1)}\Vert  f_{j'}\Vert_{L^{p,\infty}}\,.
\end{split}
\end{equation*}
For $t2^{2j}\geq1$, let
\begin{align*}
\begin{split}
A^{p,m}= \left \{
\begin{array}{ll}
\sup_{j_t\in\mathbb{Z}}\sum_{j\geq j_t}2^{2(j-j_t)mq}2^{jq(\frac np-1)}\sup_{\lambda>0}\lambda^q\lvert\{x:f_{j,j_t}(x)>\lambda\}\rvert^{\frac qp}, & 1\leq q<\infty;\\
\sup_{j_t\in\mathbb{Z},j\geq j_t}2^{2(j-j_t)m}2^{j(\frac np-1)}\sup_{\lambda>0}\lambda\lvert\{x:f_{j,j_t}(x)>\lambda\}\rvert^{\frac 1p}, & q=\infty.
\end{array}
\right.
\end{split}
\end{align*}

Similarly, we have
$A^{p,m}\lesssim\Vert f\Vert_{{\dot{B}}^{\frac{n}{p}-1, q}_{p,\infty}}$.



\section{ Proof of Theorem \ref{para} }\label{sec4}
In this section, we aim to prove the boundedness of paraproduct flow $P_{l}(u,v)$ and $P_{l,l',l''}(u,v)$. $P_{l}(u,v)$ is more easy to prove. Thus we consider only $P_{l,l',l''}(u,v)$.
The Fourier transform of the paraproduct is on the ring.
Note that
$${\rm Supp} \widehat{P_{j-2}u} \subset \Big\{\left\vert\xi_i \right\vert\leq \frac{\pi}{3} \cdot 2^{j}, \forall i=1,\cdots, n\Big\},$$
$$ {\rm Supp} \widehat{Q^{\epsilon}_{j} u} \subset
\Big\{\vert\xi_i\vert\leq \frac{4\pi}{3} \cdot 2^{j}, \mbox{ if } \epsilon_i=0;
\frac{2\pi}{3} \cdot 2^{j}\leq
\vert\xi_i\vert \leq \frac{8\pi}{3} \cdot 2^{j}, \mbox{ if } \epsilon_i=1\Big\}.$$
Hence the support of the Fourier transform of $P_{j-2}u Q^{\epsilon}_{j} v$ is contained in a ring:
$$\Big\{\vert\xi_i\vert\leq \frac{5\pi}{3} \cdot 2^{j}, \mbox{ if } \epsilon_i=0;
\frac{\pi}{3} \cdot 2^{j}\leq
\vert\xi_i\vert\leq 3\pi \cdot 2^{j}, \mbox{ if } \epsilon_i=1\Big\}.$$
Hence
\begin{align*}
&f^{\epsilon'}_{j',k'}(t)= \left\langle P_{l,l',l''}(u,v), \phi^{\epsilon'}_{j', k'}\right\rangle\\
&= \left\langle \int^{t}_{0} e^{(t-s)\Delta} \partial_{x_l}\partial_{x_{l'}}\partial_{x_{l''}}(-\Delta)^{-1} \Big\{\!\!\! \!\!
\sum\limits_{\vert j-j'\vert\leq 3, \epsilon, k, k''}
u^{0}_{j-2,k}(s) v^{\epsilon}_{j,k''}(s) \phi^{0}_{j-2,k}(x) \phi^{\epsilon}_{j,k''}(x)\Big\} ds, \phi^{\epsilon'}_{j', k'}\right\rangle.
\end{align*}
According to \cite{YYH}, we can get the following priori estimations:
\begin{lemma}\label{le:4.2}
Given $\alpha\in \mathbb{N}^{n}, N\geq n+1, k\in \mathbb{Z}^{n}$ and $0\leq s\leq t$.
For $t2^{2j}\leq 1, \epsilon\neq 0$ and $ N\geq n+1$, we have
$$ \begin{array}{c} \left\vert e^{(t-s)\Delta}\partial^{\alpha} \phi^{\epsilon}_{j,k}\right\vert \lesssim   2^{(\frac{n}{2}+\vert\alpha\vert)j} \left(1+\vert
2^{j}x-k\vert\right)^{-N},\; \forall\; 0\leq \vert\alpha\vert\leq 2,\\
\left\vert e^{(t-s)\Delta}\partial^{\alpha} (-\Delta)^{-1} \phi^{\epsilon}_{j,k}\right\vert \lesssim   2^{(\frac{n}{2}+\vert\alpha\vert-2)j}
\left(1+\vert2^{j}x-k\vert\right)^{-N}\!\!\!,\!\!\;\forall \;3\leq \vert\alpha\vert\leq 5.
\end{array}$$

For $t2^{2j}\geq 1, \epsilon\neq 0$ and $ N\geq n+1$, we have
$$ \begin{array}{c} \left\vert e^{(t-s)\Delta}\partial^{\alpha} \phi^{\epsilon}_{j,k}\right\vert \lesssim  e^{-c(t-s) 2^{2j}} 2^{(\frac{n}{2}+\vert\alpha\vert)j} \left(1+\vert
2^{j}x-k\vert\right)^{-N},\; \forall\; 0\leq \vert\alpha\vert\leq 2,\\
\left\vert e^{(t-s)\Delta}\partial^{\alpha} (-\Delta)^{-1} \phi^{\epsilon}_{j,k}\right\vert \lesssim  e^{-c(t-s) 2^{2j}} 2^{(\frac{n}{2}+\vert\alpha\vert-2)j}
\left(1+\vert2^{j}x-k\vert\right)^{-N}\!\!\!,\!\!\;\forall \;3\leq \vert\alpha\vert\leq 5.
\end{array}$$
\end{lemma}
By Lemma \ref{le:4.2} , $\phi^{0}_{j-2,k},\phi^{\epsilon}_{j,k''}\in S(\mathbb{R}^n)$ and lemma \ref{le:2.5}, for $t2^{2j^{\prime}}\geq 1$,
\begin{align*}
\left\vert f^{\epsilon'}_{j',k'}(t)\right\vert \lesssim &  2^{\frac{nj'}{2}+j'} \int ^{t}_{0}
\sum\limits_{ \vert j-j'\vert\leq 3, \epsilon, k, k''}
\left\vert u^{0}_{j-2,k}(s)\right\vert \left\vert v^{\epsilon}_{j,k''}(s)\right\vert
e^{-c(t-s)2^{2j'} } ds \\
& \times (1+\vert 4k-k''\vert)^{-N} (1+\vert 2^{j-j'}k'-k''\vert)^{-N}\,;
\end{align*}
for $t2^{2j^{\prime}}< 1$,
\begin{align*}
\left\vert f^{\epsilon'}_{j',k'}(t)\right\vert \lesssim &  2^{\frac{nj'}{2}+j'} \int ^{t}_{0}
\sum\limits_{ \vert j-j'\vert\leq 3, \epsilon, k, k''}
\left\vert u^{0}_{j-2,k}(s)\right\vert \left\vert v^{\epsilon}_{j,k''}(s)\right\vert
 ds \\
& \times (1+\vert 4k-k''\vert)^{-N} (1+\vert 2^{j-j'}k'-k''\vert)^{-N}\,.
\end{align*}

Here we only discuss the case of $1\leq q<\infty$, and we can deal with the case of $q=\infty$ in the same way.
So we omit the proof for $q=\infty$.
Firstly, we consider the indices where $0<t2^{2j'}\leq 1$. We have $0<s\leq t \leq 2^{-2j'}$ and
$\sup\limits_{k\in \mathbb{Z}^n }  \vert u^{0}_{j-2,k}(s)\vert
\lesssim  2^{-\frac{nj}{2}} 2^{j(1-2m')} s^{-m'}$, from which we can deduce

\begin{align*}
\left\vert f^{\epsilon'}_{j',k'}(t)\right\vert \lesssim &  2^{2j'(1-m')} \int ^{t}_{0}
\sum\limits_{ \vert j-j'\vert\leq 3, \epsilon, k''}
\left\vert v^{\epsilon}_{j,k''}(s)\right\vert
(1+\vert 2^{j-j'}k'-k''\vert)^{-N} s^{-m'} ds
.
\end{align*}
According to Lemma \ref{le:2.6}, for any $j_s\in\mathbb{Z}$, let $\sup_{2^{-2j_s}\leq s<2^{2-2j_s}} v_j(s,x)=v_{j,j_s}(x)$, that denotes
\begin{align*}
f_{j'}(t,x)\lesssim2^{j'(2-2m')}\sum_{t2^{2j_s}\geq1}\int^{2^{2-2j_s}}_{2^{-2j_s}}\sum_{\lvert j-j'\rvert\leq3}s^{-m'}dsM(v_{j,j_s})(x) .
\end{align*}
Hence

\begin{align*}
f_{j',j_t}(x)=\sup_{2^{-2j_t}\leq t<2^{2-2j_t}}f_{j'}(t,x)\lesssim\sum_{\lvert j-j'\rvert\leq3}\sum_{j_s\geq j_t}2^{(j'-j_s)(2-2m')}M(v_{j,j_s})(x).
\end{align*}
For $0<\delta<2-4m'$ we obtain
\begin{align*}
&\sup_{j_t\in\mathbb{Z}}\sum_{j'\leq j_t}(A^{p, m'}_{j',j_t})^q\\ =&\sup_{j_t\in\mathbb{Z}}\sum_{j'\leq j_t}(2^{2(j'-j_t)m'}2^{j'(\frac
np-1)}\sup_{\lambda>0}\lambda\lvert\{x:f_{j',j_t}(x)>\lambda\}\rvert^{\frac 1p})^q\\
\lesssim&\sup_{j_t\in\mathbb{Z}}\sum_{j'\leq j_t}(2^{2(j'-j_t)m'}2^{j'(\frac np-1)}\sup_{\lambda>0}\lambda\lvert\{x:\sum_{\lvert j-j'\rvert\leq3}\sum_{j_s\geq
j_t}2^{(j'-j_s)(2-2m')}M(v_{j,j_s})(x)>\lambda\}\rvert^{\frac 1p})^q\\
\lesssim&\sup_{j_t\in\mathbb{Z}}\sum_{j'\leq j_t}(\sum_{\lvert j-j'\rvert\leq3}\sum_{j_s\geq j_t}2^{2(j'-j_t)m'p}2^{j'(
n-p)}\\&\times\sup_{\lambda>0}\lambda^p\lvert\{x:2^{(j'-j_s)(2-2m')}M(v_{j,j_s})(x)>\frac{2^{-\delta j_s}}{\sum_{j_s\geq j_t}7\times2^{-\delta j_s}}\lambda\}\rvert)^{\frac
qp}\\
\lesssim&\sup_{j_t\in\mathbb{Z}}\sum_{j'\leq j_t}[\sum_{\lvert j-j'\rvert\leq3}\sum_{j_s\geq j_t}2^{2(j'-j_t)m'p}2^{j'( n-p)}\\&\times2^{(j'-j_s)(2-2m')p}(\frac{\sum_{j_s\geq
j_t}2^{-\delta j_s}}{2^{-\delta j_s}})^p\sup_{\lambda>0}\lambda^p\lvert\{x:M(v_{j,j_s})(x)>\lambda\}\rvert]^{\frac qp}\\
\lesssim&\sup_{j_t\in\mathbb{Z}}\sum_{j'\leq j_t}\sum_{\lvert j-j'\rvert\leq3}[\sum_{j_s\geq j_t}2^{(j_s-j_t)(2m'+\delta)p}2^{j'(
n-p)}2^{(j'-j_s)2p}\sup_{\lambda>0}\lambda^p\lvert\{x:v_{j,j_s}(x)>\lambda\}\rvert]^{\frac qp}\\
\lesssim&\sup_{j_t\in\mathbb{Z}}\sum_{j\leq j_t+3}[\sum_{j_s\geq j_t}2^{(j_s-j_t)(2m'+\delta)p}2^{j(
n-p)}2^{(j-j_s)2p}\sup_{\lambda>0}\lambda^p\lvert\{x:v_{j,j_s}(x)>\lambda\}\rvert]^{\frac qp}.
\end{align*}
For $q\leq p$, since $\delta<2-4m'$, we have
\begin{align*}
&\sup_{j_t\in\mathbb{Z}}\sum_{j'\leq j_t}(A^{p, m'}_{j',j_t})^q\\ \lesssim&\sup_{j_t\in\mathbb{Z}}\sum_{j\leq j_t+3}\sum_{j_s\geq j_t}[2^{(j_s-j_t)(2m'+\delta)p}2^{j(
n-p)}2^{(j-j_s)2p}\sup_{\lambda>0}\lambda^p\lvert\{x:v_{j,j_s}(x)>\lambda\}\rvert]^{\frac qp}\\
\lesssim&\sup_{j_t\in\mathbb{Z}}\sum_{j_s\geq j_t}\sum_{j\leq j_s}[2^{(j_s-j_t)(4m'-2+\delta)p}2^{j(
n-p)}2^{(j-j_s)2pm'}\sup_{\lambda>0}\lambda^p\lvert\{x:v_{j,j_s}(x)>\lambda\}\rvert]^{\frac qp}\\
\lesssim&\sup_{j_s\in\mathbb{Z}}\sum_{j\leq j_s}[2^{j( n-p)}2^{(j-j_s)2pm'}\sup_{\lambda>0}\lambda^p\lvert\{x:v_{j,j_s}(x)>\lambda\}\rvert]^{\frac qp}.
\end{align*}
For $q>p$, let $0<\delta'<(2-4m'-\delta)p$ and use H\"older inequality to $j_s$, we get
\begin{align*}
&\sup_{j_t\in\mathbb{Z}}\sum_{j'\leq j_t}(A^{p, m'}_{j',j_t})^q\\ \lesssim&\sup_{j_t\in\mathbb{Z}}\sum_{j\leq j_t+3}\sum_{j_s\geq
j_t}[2^{\delta'j_s}2^{(j_s-j_t)(2m'+\delta)p}2^{j( n-p)}2^{(j-j_s)2p}\sup_{\lambda>0}\lambda^p\lvert\{x:v_{j,j_s}(x)>\lambda\}\rvert]^{\frac qp}\\&\times(\sum_{j_s\geq
j_t}2^{-\delta' j_s\frac q{q-p}})^{\frac qp-1}\\
\lesssim&\sup_{j_t\in\mathbb{Z}}\sum_{j\leq j_t+3}\sum_{j_s\geq j_t}[2^{\delta'(j_s-j_t)}2^{(j_s-j_t)(2m'+\delta)p}2^{j(
n-p)}2^{(j-j_s)2p}\sup_{\lambda>0}\lambda^p\lvert\{x:v_{j,j_s}(x)>\lambda\}\rvert]^{\frac qp}.
\end{align*}
For $\delta'<(2-4m'-\delta)p$, similarly, we have
\begin{align*}
&\sup_{j_t\in\mathbb{Z}}\sum_{j'\leq j_t}(A^{p, m'}_{j',j_t})^q\lesssim\sup_{j_s\in\mathbb{Z}}\sum_{j\leq j_s}[2^{j(
n-p)}2^{(j-j_s)2pm'}\sup_{\lambda>0}\lambda^p\lvert\{x:v_{j,j_s}(x)>\lambda\}\rvert]^{\frac qp}.
\end{align*}

For the case of $1\leq t2^{2j'}\leq 4$, the proof is more easy. So we omit it.
For $t2^{2j'}\geq 4$, we write the integration of $s$ as the sum of three terms: $\int ^{2^{-2j'} } _{0} $, $\int ^{\frac{t}{2}}_{2^{-2j'}}$ and
$\int^{t}_{\frac{t}{2}}$.
We write the respective terms by $f^{\epsilon',1}_{j',k'}(t)$, $f^{\epsilon',2}_{j',k'}(t)$ and $f^{\epsilon',3}_{j',k'}(t)$. Hence
$$\left\vert f^{\epsilon'}_{j',k'}(t)\right\vert\lesssim \left\vert f^{\epsilon',1}_{j',k'}(t)\right\vert + \left\vert f^{\epsilon',2}_{j',k'}(t)\right\vert+ \left\vert
f^{\epsilon',3}_{j',k'}(t)\right\vert.$$
We consider the norm of each of the three parts in Besov-Lorentz space ${ ^{m'}_{m} \dot{B}}^{\frac{n}{p}-1,
q}_{p,\infty}$.
For the first part,
$\sup\limits_{k\in \mathbb{Z}^n }  \vert u^{0}_{j-2,k}(s)\vert
\lesssim  2^{-\frac{nj}{2}} 2^{j(1-2m')} s^{-m'}$.
Just like in the case of $t2^{2j'}<1$. We obtain that
\begin{align*}
f^1_{j'}(t,x)\lesssim&\sum_{\lvert j-j'\rvert\leq3}\sum_{j_s\geq j'}e^{-ct2^{2j'}}2^{(j-j_s)(2-2m')}M(v_{j,j_s})(x);\\
f^1_{j',j_t}(x)\lesssim&\sum_{\lvert j-j'\rvert\leq3}\sum_{j_s\geq j'}e^{-c2^{2(j'-j_t)}}2^{(j-j_s)(2-2m')}M(v_{j,j_s})(x).
\end{align*}
Hence, for $0<\delta<2-4m'$ we have
\begin{align*}
&\sup_{j_t\in\mathbb{Z}}\sum_{j'\geq j_t}(A^{p, m,1}_{j',j_t})^q\\ =:&\sup_{j_t\in\mathbb{Z}}\sum_{j'\geq j_t}(2^{2(j'-j_t)m}2^{j'(\frac
np-1)}\sup_{\lambda>0}\lambda\lvert\{x:f^1_{j',j_t}(x)>\lambda\}\rvert^{\frac 1p})^q\\
\lesssim&\sup_{j_t\in\mathbb{Z}}\sum_{j'\geq j_t}(2^{2(j'-j_t)m}2^{j'(\frac np-1)}\\&\times\sup_{\lambda>0}\lambda\lvert\{x:\sum_{\lvert j-j'\rvert\leq3}\sum_{j_s\geq
j'}e^{-c2^{2(j'-j_t)}}2^{(j-j_s)(2-2m')}M(v_{j,j_s})(x)>\lambda\}\rvert^{\frac 1p})^q\\
\lesssim&\sup_{j_t\in\mathbb{Z}}\sum_{j'\geq j_t}(\sum_{\lvert j-j'\rvert\leq3}\sum_{j_s\geq
j'}2^{2(j'-j_t)mp}2^{j'(n-p)}\\&\times\sup_{\lambda>0}\lambda^p\lvert\{x:e^{-c2^{2(j'-j_t)}}2^{(j-j_s)(2-2m')}M(v_{j,j_s})(x)>\frac{2^{-\delta j_s}}{\sum_{j_s\geq
j'}7\times2^{-\delta j_s}}\lambda\}\rvert)^{\frac qp}\\
\lesssim&\sup_{j_t\in\mathbb{Z}}\sum_{j'\geq j_t}(\sum_{\lvert j-j'\rvert\leq3}\sum_{j_s\geq j'}2^{2(j'-j_t)mp}2^{j'(n-p)}\\&\times
e^{-cp2^{2(j'-j_t)}}2^{(j-j_s)(2-2m')p}2^{\delta(j_s-j')p}\sup_{\lambda>0}\lambda^p\lvert\{x:M(v_{j,j_s})(x)>\lambda\}\rvert)^{\frac qp}\\
\lesssim&\sup_{j_t\in\mathbb{Z}}\sum_{j\geq j_t-3}(\sum_{j_s\geq j-3}2^{2(j-j_t)mp}2^{j(n-p)}\\&\times
e^{-\tilde{c}p2^{2(j-j_t)}}2^{(j-j_s)(2-2m')p}2^{\delta(j_s-j)p}\sup_{\lambda>0}\lambda^p\lvert\{x:v_{j,j_s}(x)>\lambda\}\rvert)^{\frac qp}.
\end{align*}
For $q\leq p$, since $\delta<2-4m'$, we have
\begin{align*}
&\sup_{j_t\in\mathbb{Z}}\sum_{j'\geq j_t}(A^{p, m,1}_{j',j_t})^q\\ \lesssim&\sup_{j_t\in\mathbb{Z}}\sum_{j\geq j_t-3}\sum_{j_s\geq j-3}(2^{2(j-j_t)mp}2^{j(n-p)}\\&\times
e^{-\tilde{c}p2^{2(j-j_t)}}2^{(j-j_s)(2-2m')p}2^{\delta(j_s-j)p}\sup_{\lambda>0}\lambda^p\lvert\{x:v_{j,j_s}(x)>\lambda\}\rvert)^{\frac qp}\\
\lesssim&\sup_{j_t\in\mathbb{Z}}\sum_{j_s\geq j_t-6}\sum_{j_t-3\leq j\leq
j_s+3}(2^{2(j-j_s)m'p}2^{j(n-p)}\sup_{\lambda>0}\lambda^p\lvert\{x:v_{j,j_s}(x)>\lambda\}\rvert\\&\times 2^{(j_t-j_s)(2-4m'-\delta)p})^{\frac qp}
\\\lesssim&\sup_{j_s\in\mathbb{Z}}\sum_{ j\leq j_s}(2^{2(j-j_s)m'p}2^{j(n-p)}\sup_{\lambda>0}\lambda^p\lvert\{x:v_{j,j_s}(x)>\lambda\}\rvert)^{\frac qp}.
\end{align*}
For $q>p$, let $0<\delta'<(2-4m'-\delta)p$ and use H\"older inequality to $j_s$, we have
\begin{align*}
&\sup_{j_t\in\mathbb{Z}}\sum_{j'\geq j_t}(A^{p, m,1}_{j',j_t})^q\\ \lesssim&\sup_{j_t\in\mathbb{Z}}\sum_{j\geq j_t-3}\sum_{j_s\geq j-3}(2^{\delta' j_s}2^{2(j-j_t)mp}2^{j(n-p)}
e^{-\tilde{c}p2^{2(j-j_t)}}2^{(j-j_s)(2-2m')p}2^{\delta(j_s-j)p}\\&\sup_{\lambda>0}\lambda^p\lvert\{x:v_{j,j_s}(x)>\lambda\}\rvert)^{\frac qp}\times(\sum_{j_s\geq
j}2^{-\delta' j_s\frac q{q-p}})^{\frac qp-1}\\
\lesssim&\sup_{j_t\in\mathbb{Z}}\sum_{j\geq j_t-3}\sum_{j_s\geq j-3}(2^{\delta' (j_s-j)}2^{2(j-j_t)mp}2^{j(n-p)}
e^{-\tilde{c}p2^{2(j-j_t)}}2^{(j-j_s)(2-2m')p}2^{\delta(j_s-j)p}\\&\sup_{\lambda>0}\lambda^p\lvert\{x:v_{j,j_s}(x)>\lambda\}\rvert)^{\frac qp}.
\end{align*}
For $\delta'<(2-4m'-\delta)p$, similarly, we have
\begin{align*}
&\sup_{j_t\in\mathbb{Z}}\sum_{j'\leq j_t}(A^{p, m',1}_{j',j_t})^q\lesssim\sup_{j_s\in\mathbb{Z}}\sum_{j\leq j_s}[2^{j(
n-p)}2^{(j-j_s)2pm'}\sup_{\lambda>0}\lambda^p\lvert\{x:v_{j,j_s}(x)>\lambda\}\rvert]^{\frac qp}.
\end{align*}


For the second part,
$\sup\limits_{k\in \mathbb{Z}^n }  \vert u^{0}_{j-2,k}(s)\vert
\lesssim  2^{-\frac{nj}{2}} s^{-\frac{1}{2}}$.
We obtain
\begin{align*}
f^2_{j'}(t,x)\lesssim&\sum_{\lvert j-j'\rvert\leq3}\sum_{j_s\geq j',t2^{2j_s}\geq2}2^{j'-j_s}e^{-\tilde{c}t2^{2j'}}M(v_{j,j_s})(x);\\
f^2_{j',j_t}(x)\lesssim&\sum_{\lvert j-j'\rvert\leq3}\sum_{j'\geq j_s\geq j_t}2^{j'-j_s}e^{-\tilde{c}2^{2(j'-j_t)}}M(v_{j,j_s})(x).
\end{align*}
Hence
\begin{align*}
&\sup_{j_t\in\mathbb{Z}}\sum_{j'\geq j_t}(A^{p, m,2}_{j',j_t})^q\\ =:&\sup_{j_t\in\mathbb{Z}}\sum_{j'\geq j_t}(2^{2(j'-j_t)m}2^{j'(\frac
np-1)}\sup_{\lambda>0}\lambda\lvert\{x:f^2_{j',j_t}(x)>\lambda\}\rvert^{\frac 1p})^q\\
\lesssim&\sup_{j_t\in\mathbb{Z}}\sum_{j'\geq j_t}(2^{2(j'-j_t)m}2^{j'(\frac np-1)}\\&\sup_{\lambda>0}\lambda\lvert\{x:\sum_{\lvert j-j'\rvert\leq3}\sum_{j'\geq j_s\geq
j_t}2^{j'-j_s}e^{-\tilde{c}2^{2(j'-j_t)}}M(v_{j,j_s})(x)>\lambda\}\rvert^{\frac 1p})^q\\
\lesssim&\sup_{j_t\in\mathbb{Z}}\sum_{j'\geq j_t}(2^{2(j'-j_t)m}2^{j'(\frac np-1)}\sum_{\lvert j-j'\rvert\leq3}\sum_{j'\geq j_s\geq
j_t}\\&\sup_{\lambda>0}\lambda\lvert\{x:2^{j'-j_s}e^{-\tilde{c}2^{2(j'-j_t)}}M(v_{j,j_s})(x)>\frac1{7\times (j'-j_t)}\lambda\}\rvert^{\frac 1p})^q\\
\lesssim&\sup_{j_t\in\mathbb{Z}}\sum_{j'\geq j_t}\sum_{\lvert j-j'\rvert\leq3}\sum_{j'\geq j_s\geq
j_t}(2^{2(j'-j_t)mp}2^{j'(n-p)}\\&2^{(j'-j_s)p}e^{-\tilde{c}p2^{2(j'-j_t)}}(j'-j_t)^{p+\lvert1-\frac
pq\rvert}\sup_{\lambda>0}\lambda^p\lvert\{x:M(v_{j,j_s})(x)>\lambda\}\rvert)^{\frac qp}\\
\lesssim&\sup_{j_t\in\mathbb{Z}}\sum_{j\geq j_t-3}\sum_{j+3\geq j_s\geq j_t}(2^{2(j-j_t)mp}2^{j(n-p)}\\&2^{(j-j_s)p}e^{-\tilde{c}p2^{2(j-j_t)}}(j-j_t)^{p+\lvert1-\frac
pq\rvert}\sup_{\lambda>0}\lambda^p\lvert\{x:M(v_{j,j_s})(x)>\lambda\}\rvert)^{\frac qp}\\
\lesssim&\sup_{j_t\in\mathbb{Z}}\sum_{j_s\geq j_t}2^{(j_t-j_s)p}\sum_{j\geq j_s-3}(2^{2(j-j_t)mp}2^{j(n-p)}\\&2^{(j-j_t)p}e^{-\tilde{c}p2^{2(j-j_t)}}(j-j_t)^{p+\lvert1-\frac
pq\rvert}\sup_{\lambda>0}\lambda^p\lvert\{x:M(v_{j,j_s})(x)>\lambda\}\rvert)^{\frac qp}\\
\lesssim&\sup_{j_s\in\mathbb{Z}}\sum_{j\geq j_s}(2^{2(j-j_t)mp}2^{j(n-p)}\sup_{\lambda>0}\lambda^p\lvert\{x:v_{j,j_s}(x)>\lambda\}\rvert)^{\frac qp} .
\end{align*}


For the third part,
$\sup\limits_{k\in \mathbb{Z}^n }  \vert u^{0}_{j-2,k}(s)\vert
\lesssim  2^{-\frac{nj}{2}} t^{-\frac{1}{2}}$.
 Which implies,
\begin{align*}
&f^3_{j'}(t,x)\lesssim2^{j'}\sum\limits_{  \epsilon', k'}\sum_{\lvert j-j'\rvert\leq3}\int^{t}_{\frac t2}s^{-\frac 12}e^{-c(t-s)2^{2j'}}\sum\limits_{  \epsilon, k''}
\frac{2^{\frac n2j}\left\vert v^{\epsilon}_{j,k''}(s)\right\vert}{
(1+\vert 2^{j-j'}k'-k''\vert)^{N}} ds\chi(2^{j'}x-k').
\end{align*}
Hence
\begin{align*}
&f^3_{j',j_t}(x)=\sup_{2^{-2j_t}\leq t<2^{2-2j_t}}f_{j'}(t,x)\\\lesssim&2^{j'}\sum_{\lvert j-j'\rvert\leq3}\sum_{j_s=j_t\,or\,j_t+1}\sup_{2^{-2j_t}\leq
t<2^{2-2j_t}}\int^{2^{2-2j_s}}_{2^{-2j_s}}s^{-\frac 12}e^{-c(t-s)2^{2j'}}dsM(v_{j,j_s})(x)\\\lesssim&\sum_{\lvert
j-j'\rvert\leq3}\sum_{j_s=j_t\,or\,j_t+1}2^{j_s-j'}M(v_{j,j_s})(x).
\end{align*}
Thus,
\begin{align*}
&\sup_{j_t\in\mathbb{Z}}\sum_{j'\geq j_t}(A^{p, m,3}_{j',j_t})^q\\ =:&\sup_{j_t\in\mathbb{Z}}\sum_{j'\geq j_t}(2^{2(j'-j_t)m}2^{j'(\frac
np-1)}\sup_{\lambda>0}\lambda\lvert\{x:f^3_{j',j_t}(x)>\lambda\}\rvert^{\frac 1p})^q\\
\lesssim&\sup_{j_t\in\mathbb{Z}}\sum_{j'\geq j_t}(2^{2(j'-j_t)m}2^{j'(\frac np-1)}\\&\sup_{\lambda>0}\lambda\lvert\{x:\sum_{\lvert
j-j'\rvert\leq3}\sum_{j_s=j_t\,or\,j_t+1}2^{j_s-j'}M(v_{j,j_s})(x)>\lambda\}\rvert^{\frac 1p})^q\\
\lesssim&\sup_{j_t\in\mathbb{Z}}\sum_{j'\geq j_t}\sum_{\lvert
j-j'\rvert\leq3}\sum_{j_s=j_t\,or\,j_t+1}(2^{2(j'-j_t)mp}2^{j'(n-p)}2^{(j_s-j')p}\\&\sup_{\lambda>0}\lambda^p\lvert\{x:M(v_{j,j_s})(x)>\lambda\}\rvert)^{\frac qp}\\
\lesssim&\sup_{j_s\in\mathbb{Z}}\sum_{j\geq j_s-3}(2^{2(j-j_s)mp}2^{j(n-p)}2^{(j_s-j)p}\sup_{\lambda>0}\lambda^p\lvert\{x:M(v_{j,j_s})(x)>\lambda\}\rvert)^{\frac qp}\\
\lesssim&\sup_{j_s\in\mathbb{Z}}\sum_{j\geq j_s}(2^{2(j-j_s)mp}2^{j(n-p)}\sup_{\lambda>0}\lambda^p\lvert\{x:v_{j,j_s}(x)>\lambda\}\rvert)^{\frac qp} .
\end{align*}



\section{ Proof of Theorem \ref{cou}}\label{sec5}

In this section we will consider the boundedness of couple flow $C_{l}(u,v)$ and $C_{l,l',l''}(u,v)$.
Just like in Section \ref{sec4} we consider only $C_{l,l',l''}(u,v)$.
Note that the following properties on the support of Fourier transform:
$$ {\rm Supp} \widehat{\phi^{\epsilon}_{j,k} \phi^{\tilde{\epsilon}}_{j,k''}} \subset
\Big\{\vert\xi_i\vert\leq \frac{16\pi}{3} \cdot 2^{j}, \forall i=1,\cdots,n\Big\},$$
$$ {\rm Supp} \widehat{\phi^{\epsilon'}_{j', k'}} \subset
\Big\{\vert\xi_i\vert\leq \frac{4\pi}{3} \cdot 2^{j'}, \mbox{ if } \epsilon_i=0;
\frac{2\pi}{3} \cdot 2^{j'}\leq
\vert\xi_i\vert\leq \frac{8\pi}{3} \cdot 2^{j'}, \mbox{ if } \epsilon_i=1\Big\}.$$
Hence
\begin{align*}
&f^{\epsilon'}_{j',k'}(t)= \left\langle C_{l,l',l''}(u,v), \phi^{\epsilon'}_{j', k'}\right\rangle\\
&= \left\langle \int^{t}_{0} e^{(t-s)\Delta} \partial_{x_l}\partial_{x_{l'}}\partial_{x_{l''}}(-\Delta)^{-1} \Big\{\!\!\! \!\!
 \sum\limits_{j\geq  j'-2,\epsilon,\tilde{\epsilon}, k,k''}  u^{\epsilon}_{j,k}(s) v^{\tilde{\epsilon}}_{j,k''}(s)\phi^{\epsilon}_{j,k}(x) \phi^{\tilde{\epsilon}}_{j,k''}(x)
\Big\} ds, \phi^{\epsilon'}_{j', k'}\right\rangle.
\end{align*}
By Lemma \ref{le:4.2} we can also get
\begin{align*}
\vert f^{\epsilon'}_{j',k'}(t)\vert
\lesssim & \int^{t}_{0}  \sum\limits_{j\geq  j'-2 } \sum\limits_{\epsilon,\tilde{\epsilon}, k,k''} \left\vert u^{\epsilon}_{j,k}(s)\right\vert \left\vert
v^{\tilde{\epsilon}}_{j,k''}(s)\right\vert
\left(1+\vert k-k''\vert\right)^{-N} \\
&\times \left(1+ \left\vert2^{j'-j}k-k'\right\vert\right)^{-N} 2^{(\frac{n}{2}+1)j'} e^{-c(t-s) 2^{2j'}}ds.
\end{align*}



Similar to Section \ref{sec4}, here we only discuss the case of $1\leq q<\infty$ and omit the case of $q=\infty$.
Firstly, we consider first the case where $0<t2^{2j'}\leq 1$.
We write the integration $\int^{t}_{0}$ as the sum of three parts:
$\int^{2^{-2j}}_{0}$, $\int^{\frac{t}{2}}_{2^{-2j}}$ and $\int^{t}_{\frac{t}{2}}$.
If $t\leq2^{-2j}$, we just consider the case of $\int^{2^{-2j}}_{0}$. For this case, we have
$ |v^{\epsilon}_{j,k}(s)|\leq (s2^{2j})^{-m'} 2^{(1-\frac{n}{2})j}.$
\begin{align*}
\vert f^{\epsilon',1}_{j',k'}(t)\vert
\lesssim &   \int^{2^{-2j}}_{0}  \sum\limits_{j\geq  j'-2 } \sum\limits_{\epsilon,\tilde{\epsilon}, k,k''}
\left\vert u^{\epsilon}_{j,k}(s)\right\vert
\left(1+\vert k-k''\vert\right)^{-N} \\
&\times \left(1+ \left\vert2^{j'-j}k-k'\right\vert\right)^{-N} 2^{(\frac{n}{2}+1)j'} (s2^{2j})^{-m'} 2^{(1-\frac{n}{2})j} ds\\
\lesssim &  \sum\limits_{j\geq  j'-2 } \int^{2^{-2j}}_{0}   \sum\limits_{\epsilon,k}
\frac{\left\vert u^{\epsilon}_{j,k}(s)\right\vert}
{\left(1+ \left\vert2^{j'-j}k-k'\right\vert\right)^{N} }  (s2^{2j})^{-m'} 2^{(1+\frac{n}{2})(j'-j)} \times2^{2j}  ds.\\
\end{align*}
For any $j_s\in\mathbb{Z}$, let $\sup_{2^{-2j_s}\leq s<2^{2-2j_s}} u(s,x)=u_{j_s}(x)$
and $\sup_{2^{-2j_s}\leq s<2^{2-2j_s}} u^{\epsilon}_{j,k}(s)=(u^{\epsilon}_{j,k})_{j_s}$.
According to lemma \ref{le:2.6}, we obtain
\begin{align*}
f^1_{j',j_t}(x)=&\sup_{2^{-2j_t}\leq t<2^{2-2j_t}}f^1_{j'}(t,x)=\sup_{2^{-2j_t}\leq t<2^{2-2j_t}}2^{\frac{n}{2}j'} \sum\limits_{(\epsilon',k')\in
\Gamma}|f^{\epsilon',1}_{j',k'}(t)|\chi(2^{j'}x-k')\\
\lesssim &   \sup_{2^{-2j_t}\leq t<2^{2-2j_t}} \sum\limits_{(\epsilon',k')\in \Gamma} \sum\limits_{j\geq  j'-2 } \int^{2^{-2j}}_{0} 2^{\frac
n2(j'-j)}2^{2j}\sum\limits_{\epsilon,k}\frac{2^{\frac n2j}\left\vert u^{\epsilon}_{j,k}(s)\right\vert}
{\left(1+ \left\vert2^{j'-j}k-k'\right\vert\right)^{N} }\\&\times2^{(\frac n2+1)(j'-j)} (s2^{2j})^{-m'}
 ds\chi(2^{j'}x-k')\\
 \lesssim &   \sup_{2^{-2j_t}\leq t<2^{2-2j_t}} \sum\limits_{(\epsilon',k')\in \Gamma} \sum\limits_{j\geq  j'-2 } \sum_{j_s\geq j+1}\int^{2^{2-2j_s}}_{2^{-2j_s}} 2^{\frac
 n2(j'-j)}2^{2j}\\&\times2^{(\frac n2+1)(j'-j)} (s2^{2j})^{-m'}
 ds\sum\limits_{\epsilon,k}\frac{2^{\frac n2j}\left\vert (u^{\epsilon}_{j,k})_{j_s}\right\vert}
{\left(1+ \left\vert2^{j'-j}k-k'\right\vert\right)^{N} }\chi(2^{j'}x-k')\\
 \lesssim &   \sup_{2^{-2j_t}\leq t<2^{2-2j_t}} \sum\limits_{(\epsilon',k')\in \Gamma} \sum\limits_{j\geq  j'-2 } \sum_{j_s\geq j+1}\int^{2^{2-2j_s}}_{2^{-2j_s}} 2^{\frac
 n2(j'-j)}2^{2j}\\&\times2^{(\frac n2+1)(j'-j)} (s2^{2j})^{-m'}
 ds M(u_{j,j_s})(x)2^{n(j-j')}\chi(2^{j'}x-k')\\
\lesssim & \sup_{2^{-2j_t}\leq t<2^{2-2j_t}} \sum\limits_{j\geq  j'-2 } \sum_{j_s\geq j+1}\int^{2^{2-2j_s}}_{2^{-2j_s}} 2^{j'+j}(s2^{2j})^{-m'}dsM(u_{j,j_s})(x)\\
\lesssim &   \sum\limits_{j\geq  j'-2 }\sum\limits_{j_s\geq j+1} M(u_{j,j_s})(x)2^{j'-j}2^{2(j-j_s)(1-m')} .
\end{align*}
For $s\leq2^{-2j}$, let $A^{p, m',1}_{j',j_t}=2^{2(j-j_t)m'}2^{j(\frac np-1)}\sup_{\lambda>0}\lambda|\{x:f^1_{j,j_t}(x)>\lambda\}|^{\frac 1p}$. That is to say, for $\frac
np+4m'-2<\delta<\frac np$, $0<\delta'<\delta-\frac np+2-4m'$,
\begin{align*}
&\sup_{j_t\in\mathbb{Z}}\sum_{j'\leq j_t}(A^{p, m',1}_{j',j_t})^q\\=&\sup_{j_t\in\mathbb{Z}}\sum_{j'\leq j_t}2^{2(j'-j_t)m'q}2^{j'q(\frac
np-1)}\sup_{\lambda>0}\lambda^q|\{x:f^1_{j',j_t}(x)>\lambda\}|^{\frac qp}\\
\lesssim & \sup_{j_t\in\mathbb{Z}}\sum_{j'\leq j_t}2^{2(j'-j_t)m'q}2^{j'q(\frac np-1)}\sup_{\lambda>0}\lambda^q[\sum_{j\geq j'-2}\sum_{j_s\geq
j+1}\\&|\{x:M(u_{j,j_s})(x)2^{j'-j}2^{2(j-j_s)(1-m')}>\frac{2^{-\delta j}2^{-\delta'j_s}}{\sum_{j\geq j'-2}\sum_{j_s\geq j+1}2^{-\delta j}2^{-\delta'j_s}}\lambda\}|]^{\frac
qp}\\
\lesssim & \sup_{j_t\in\mathbb{Z}}\sum_{j'\leq j_t}[2^{2(j'-j_t)m'p}2^{j'p(\frac np-1)}\sum_{j\geq j'-2}\sum_{j_s\geq
j+1}\sup_{\lambda>0}\lambda^p\\&|\{x:M(u_{j,j_s})(x)>\lambda\}|2^{(j'-j)p}2^{2p(j-j_s)(1-m')}(\frac{\sum_{j\geq j'-2}\sum_{j_s\geq j+1}2^{-\delta j}2^{-\delta'j_s}}{2^{-\delta
j}2^{-\delta'j_s}})^p]^{\frac qp}\\
\lesssim &  \sup_{j_t\in\mathbb{Z}}\sum_{j'\leq j_t}[\sum_{j\geq j'-2}\sum_{j_s\geq
j+1}\sup_{\lambda>0}\lambda^p|\{x:M(u_{j,j_s})(x)>\lambda\}|\\&\times2^{2(j'-j_t)m'p}2^{nj'}2^{-jp}2^{2p(j-j_s)(1-m')}2^{\delta (j-j')p+\delta'(j_s-j)p}]^{\frac qp}  .
\end{align*}
By lemma \ref{le:2.7}, we can get
\begin{align*}
\sup_{j_t\in\mathbb{Z}}\sum_{j'\leq j_t}(A^{p, m',1}_{j',j_t})^q\lesssim&\sup_{j_t\in\mathbb{Z}}\sum_{j'\leq j_t}(\sum_{j\geq j'-2}\sum_{j_s\geq
j+1}\sup_{\lambda>0}\lambda^p|\{x:u_{j,j_s}(x)>\lambda\}|\\&\times2^{2(j'-j_t)m'p}2^{nj'}2^{-jp}2^{2p(j-j_s)(1-m')}2^{\delta (j-j')p+\delta'(j_s-j)p})^{\frac qp}.
\end{align*}
\par For $q\leq p$, we have
\begin{align*}
\sup_{j_t\in\mathbb{Z}}\sum_{j'\leq j_t}(A^{p, m',1}_{j',j_t})^q\lesssim&\sup_{j_t\in\mathbb{Z}}\sum_{j'\leq j_t}\sum_{j\geq j'-2}\sum_{j_s\geq
j+1}(\sup_{\lambda>0}\lambda^p|\{x:u_{j,j_s}(x)>\lambda\}|\\&\times2^{2(j'-j_t)m'p}2^{nj'}2^{-jp}2^{2p(j-j_s)(1-m')}2^{\delta (j-j')p+\delta'(j_s-j)p})^{\frac qp} \\
\lesssim&\sup_{j_t\in\mathbb{Z}}\sum_{j'\leq j_t}\sum_{j_s\geq j'-1}\sum_{j\leq
j_s-1}(\sup_{\lambda>0}\lambda^p|\{x:u_{j,j_s}(x)>\lambda\}|\\&\times2^{2(j'-j_t)m'p}2^{nj'}2^{-jp}2^{2p(j-j_s)(1-m')}2^{\delta (j-j')p+\delta'(j_s-j)p})^{\frac qp}\\
\lesssim&\sup_{j_t\in\mathbb{Z}}(\sum_{j_s\leq j_t-1}\sum_{j'\leq j_s+1}+\sum_{j_s\geq j_t-1}\sum_{j'\leq j_t})\sum_{j\leq
j_s-1}(\sup_{\lambda>0}\lambda^p|\{x:u_{j,j_s}(x)>\lambda\}|\\&\times2^{2(j'-j_t)m'p}2^{nj'}2^{-jp}2^{2p(j-j_s)(1-m')}2^{\delta (j-j')p+\delta'(j_s-j)p})^{\frac qp}  .
\end{align*}
Since $2-4m'-\frac np+\delta-\delta'>0$ and $\delta<\frac np$, we can obtain
\begin{align*}
&\sup_{j_t\in\mathbb{Z}}\sum_{j_s\leq j_t-1}\sum_{j'\leq j_s+1}\sum_{j\leq
j_s-1}(\sup_{\lambda>0}\lambda^p|\{x:u_{j,j_s}(x)>\lambda\}|\\&\times2^{2(j'-j_t)m'p}2^{nj'}2^{-jp}2^{2p(j-j_s)(1-m')}2^{\delta (j-j')p+\delta'(j_s-j)p})^{\frac qp} \\
\lesssim&\sup_{j_t\in\mathbb{Z}}\sum_{j_s\leq j_t-1}\sum_{j\leq
j_s-1}(\sup_{\lambda>0}\lambda^p|\{x:u_{j,j_s}(x)>\lambda\}|\\&\times2^{2(j_s-j_t)m'p}2^{nj_s}2^{-jp}2^{2p(j-j_s)(1-m')}2^{\delta (j-j_s)p+\delta'(j_s-j)p})^{\frac qp}\\
\lesssim&\sup_{j_t\in\mathbb{Z}}\sum_{j_s\leq j_t-1}\sum_{j\leq j_s-1}\sup_{\lambda>0}\lambda^q|\{x:u_{j,j_s}(x)>\lambda\}|^{\frac qp}2^{2(j-j_s)m'q}2^{j(\frac
np-1)q}\\&\times2^{2(j_s-j_t)m'q}2^{\frac np(j_s-j)q}2^{2q(j-j_s)(1-2m')}2^{\delta (j-j_s)q+\delta'(j_s-j)q}\\
\lesssim&\sup_{j_t\in\mathbb{Z}}\sum_{j_s\leq j_t-1}\sum_{j\leq j_s}\sup_{\lambda>0}\lambda^q|\{x:u_{j,j_s}(x)>\lambda\}|^{\frac qp}2^{2(j-j_s)m'q}2^{j(\frac
np-1)q}2^{2(j_s-j_t)m'q}\\
\lesssim&\sup_{j_s\in\mathbb{Z}}\sum_{j\leq j_s}\sup_{\lambda>0}\lambda^q|\{x:u_{j,j_s}(x)>\lambda\}|^{\frac qp}2^{2(j-j_s)m'q}2^{j(\frac np-1)q}   .
\end{align*}
Hence
\begin{align*}
&\sup_{j_t\in\mathbb{Z}}\sum_{j_s\geq j_t-1}\sum_{j'\leq j_t}\sum_{j\leq
j_s-1}(\sup_{\lambda>0}\lambda^p|\{x:u_{j,j_s}(x)>\lambda\}|\\&\times2^{2(j'-j_t)m'p}2^{nj'}2^{-jp}2^{2p(j-j_s)(1-m')}2^{\delta (j-j')p+\delta'(j_s-j)p})^{\frac qp} \\
\lesssim&\sup_{j_t\in\mathbb{Z}}\sum_{j_s\geq j_t-1}\sum_{j\leq
j_s-1}(\sup_{\lambda>0}\lambda^p|\{x:u_{j,j_s}(x)>\lambda\}|\\&\times2^{nj_t}2^{-jp}2^{2p(j-j_s)(1-m')}2^{\delta (j-j_t)p+\delta'(j_s-j)p})^{\frac qp} \\
\lesssim&\sup_{j_t\in\mathbb{Z}}\sum_{j_s\geq j_t-1}\sum_{j\leq j_s-1}\sup_{\lambda>0}\lambda^q|\{x:u_{j,j_s}(x)>\lambda\}|^{\frac qp}2^{2(j-j_s)m'q}2^{j(\frac
np-1)q}\\&\times2^{(\frac np-\delta)q(j_t-j_s)}2^{q(j-j_s)(2-4m'-\frac np)}2^{\delta (j-j_s)q+\delta'(j_s-j)q} \\
\lesssim&\sup_{j_s\in\mathbb{Z}}\sum_{j\leq j_s}\sup_{\lambda>0}\lambda^q|\{x:u_{j,j_s}(x)>\lambda\}|^{\frac qp}2^{2(j-j_s)m'q}2^{j(\frac np-1)q}  .
\end{align*}
For $q>p$, apply H\"older inequality to $j$ and $j_s$. For $0<\delta''<n-\delta p$, we have
\begin{align*}
\sup_{j_t\in\mathbb{Z}}\sum_{j'\leq j_t}(A^{p, m',1}_{j',j_t})^q\lesssim&\sup_{j_t\in\mathbb{Z}}\sum_{j'\leq j_t}\sum_{j\geq j'-2}\sum_{j_s\geq
j+1}(2^{\delta''j_s}\sup_{\lambda>0}\lambda^p|\{x:u_{j,j_s}(x)>\lambda\}|\\&\times2^{2(j'-j_t)m'p}2^{nj'}2^{-jp}2^{2p(j-j_s)(1-m')}2^{\delta (j-j')p+\delta'(j_s-j)p})^{\frac
qp}\\&\times( \sum_{j\geq j'-2}\sum_{j_s\geq j+1}2^{-\delta''j_s\frac q{q-p}})^{\frac qp-1}\\
\lesssim&\sup_{j_t\in\mathbb{Z}}\sum_{j'\leq j_t}\sum_{j\geq j'-2}\sum_{j_s\geq
j+1}(2^{\delta''(j_s-j')}\sup_{\lambda>0}\lambda^p|\{x:u_{j,j_s}(x)>\lambda\}|\\&\times2^{2(j'-j_t)m'p}2^{nj'}2^{-jp}2^{2p(j-j_s)(1-m')}2^{\delta
(j-j')p+\delta'(j_s-j)p})^{\frac qp}.
\end{align*}
Since that $\delta''<n-\delta p$, similarly, we have
$$\sup_{j_t\in\mathbb{Z}}\sum_{j'\leq j_t}(A^{p, m',1}_{j',j_t})^q\lesssim\sup_{j_s\in\mathbb{Z}}\sum_{j\leq
j_s}\sup_{\lambda>0}\lambda^q|\{x:u_{j,j_s}(x)>\lambda\}|^{\frac qp}2^{2(j-j_s)m'q}2^{j(\frac np-1)q}.$$




For the second case, we have
$ |v^{\epsilon}_{j,k}(s)|\leq (s2^{2j})^{-m} 2^{(1-\frac{n}{2})j}.$
\begin{align*}
\vert f^{\epsilon',2}_{j',k'}(t)\vert
\lesssim &   \int^{\frac t2}_{2^{-2j}}  \sum\limits_{j\geq  j'-2 } \sum\limits_{\epsilon,\tilde{\epsilon}, k,k''}
\left\vert u^{\epsilon}_{j,k}(s)\right\vert
\left(1+\vert k-k''\vert\right)^{-N} \\
&\times \left(1+ \left\vert2^{j'-j}k-k'\right\vert\right)^{-N} 2^{(\frac{n}{2}+1)j'} (s2^{2j})^{-m} 2^{(1-\frac{n}{2})j}ds\\
\lesssim &  \sum\limits_{j\geq  j'-2 }\int^{\frac t2}_{2^{-2j}}   \sum\limits_{\epsilon,k}
\frac{\left\vert u^{\epsilon}_{j,k}(s)\right\vert}
{\left(1+ \left\vert2^{j'-j}k-k'\right\vert\right)^{N} }  (s2^{2j})^{-m} 2^{(1+\frac{n}{2})(j'-j)} \times2^{2j}  ds  .
\end{align*}
By lemma \ref{le:2.6}, we have
\begin{align*}
f^2_{j',j_t}(x)=&\sup_{2^{-2j_t}\leq t<2^{2-2j_t}}2^{\frac{n}{2}j'} \sum\limits_{(\epsilon',k')\in \Gamma}|f^{\epsilon',2}_{j',k'}(t)|\chi(2^{j'}x-k')\\
\lesssim &   \sup_{2^{-2j_t}\leq t<2^{2-2j_t}} \sum\limits_{(\epsilon',k')\in \Gamma} \sum\limits_{j\geq  j'-2 }\int^{\frac t2}_{2^{-2j}} 2^{\frac
n2(j'-j)}2^{2j}\sum\limits_{\epsilon,k}
\frac{2^{\frac n2j}\left\vert u^{\epsilon}_{j,k}(s)\right\vert}
{\left(1+ \left\vert2^{j'-j}k-k'\right\vert\right)^{N} } \\&\times2^{(\frac n2+1)(j'-j)} (s2^{2j})^{-m}
 ds\chi(2^{j'}x-k')\\
\lesssim & \sup_{2^{-2j_t}\leq t<2^{2-2j_t}} \sum\limits_{j\geq  j'-2 } \sum_{\frac{1-\log_2t}{2} \leq j_s< j+1}
\int^{2^{2-2j_s}}_{2^{-2j_s}}
2^{j'+j}(s2^{2j})^{-m}ds\\
&\times \sum\limits_{(\epsilon',k')\in \Gamma}\sum\limits_{\epsilon,k}
\frac{2^{\frac n2j} (u^{\epsilon}_{j,k})_{j_s}}
{\left(1+ \left\vert2^{j'-j}k-k'\right\vert\right)^{N} }\chi(2^{j'}x-k')\\
\lesssim &   \sum\limits_{j\geq  j'-2 }\sum\limits_{j_t\leq j_s\leq j}M(u_{j,j_s})(x)2^{j'-j}2^{2(j-j_s)(1-m)} .
\end{align*}
In this case, for $0<\delta<\frac np$
\begin{align*}
&\sup_{j_t\in\mathbb{Z}}\sum_{j'\leq j_t}(A^{p, m',2}_{j',j_t})^q\\=&\sup_{j_t\in\mathbb{Z}}\sum_{j'\leq j_t}2^{2(j'-j_t)m'q}2^{j'q(\frac
np-1)}\sup_{\lambda>0}\lambda^q|\{x:f^2_{j',j_t}(x)>\lambda\}|^{\frac qp}\\
\lesssim & \sup_{j_t\in\mathbb{Z}}\sum_{j'\leq j_t}2^{2(j'-j_t)m'q}2^{j'q(\frac np-1)}\sup_{\lambda>0}\lambda^q[\sum_{j\geq j'-2}\sum_{j_t\leq j_s\leq
j}\\&|\{x:M(u_{j,j_s})(x)2^{j'-j}2^{2(j-j_s)(1-m)}>\frac{2^{-\delta j}}{\sum_{j\geq j'-2}2^{-\delta j}(j-j_t+1)}\lambda\}|]^{\frac qp}\\
\lesssim & \sup_{j_t\in\mathbb{Z}}\sum_{j'\leq j_t}[2^{2(j'-j_t)m'p}2^{j'p(\frac np-1)}\sum_{j\geq j'-2}\sum_{j_t\leq j_s\leq
j}\sup_{\lambda>0}\lambda^p|\{x:M(u_{j,j_s})(x)>\lambda\}| \\&2^{(j'-j)p}2^{2p(j-j_s)(1-m)}(\frac{\sum_{j\geq j'-2}2^{-\delta j}(j-j_t+1)}{2^{-\delta j}})^p]^{\frac qp}\\
\lesssim &  \sup_{j_t\in\mathbb{Z}}\sum_{j'\leq j_t}[\sum_{j\geq j'-2}\sum_{j_t\leq j_s\leq
j}\sup_{\lambda>0}\lambda^p|\{x:M(u_{j,j_s})(x)>\lambda\}|\\&\times2^{2(j'-j_t)m'p}2^{nj'}2^{-jp}2^{2p(j-j_s)(1-m)}2^{\delta (j-j')p}(j-j_t+1)^p]^{\frac qp}\\
\lesssim &  \sup_{j_t\in\mathbb{Z}}\sum_{j'\leq j_t}[\sum_{j\geq j'-2}\sum_{j_t\leq j_s\leq j}\sup_{\lambda>0}\lambda^p|\{x:u_{j,j_s}(x)>\lambda\}|\\&\times
2^{2(j'-j_t)m'p}2^{nj'}2^{-jp}2^{2p(j-j_s)(1-m)}2^{\delta (j-j')p}(j-j_t+1)^p]^{\frac qp}  .
\end{align*}
\par For $q\leq p$, we can get
\begin{align*}
&\sup_{j_t\in\mathbb{Z}}\sum_{j'\leq j_t}(A^{p, m',2}_{j',j_t})^q\\\lesssim&\sup_{j_t\in\mathbb{Z}}\sum_{j'\leq j_t}\sum_{j\geq j'-2}\sum_{j_t\leq j_s\leq
j}(\sup_{\lambda>0}\lambda^p|\{x:u_{j,j_s}(x)>\lambda\}| \\&\times2^{2(j'-j_t)m'p}2^{nj'}2^{-jp}2^{2p(j-j_s)(1-m)}2^{\delta (j-j')p}(j-j_t+1)^p)^{\frac qp} \\
\lesssim&\sup_{j_t\in\mathbb{Z}}\sum_{j_s\geq j_t}\sum_{j'\leq j_t}\sum_{j\geq j_s}(\sup_{\lambda>0}\lambda^p|\{x:u_{j,j_s}(x)>\lambda\}| \\&\times
2^{2(j'-j_t)m'p}2^{nj'}2^{-jp}2^{2p(j-j_s)(1-m)}2^{\delta (j-j')p}(j-j_t+1)^p)^{\frac qp}  .
\end{align*}
Since $\delta<\frac np$, we have
\begin{align*}
&\sup_{j_t\in\mathbb{Z}}\sum_{j'\leq j_t}(A^{p, m',2}_{j',j_t})^q\\
\lesssim&\sup_{j_t\in\mathbb{Z}}\sum_{j_s\geq j_t}\sum_{j\geq j_s}(\sup_{\lambda>0}\lambda^p|\{x:u_{j,j_s}(x)>\lambda\}|\\&\times 2^{nj_t}2^{-jp}2^{2p(j-j_s)(1-m)}2^{\delta
(j-j_t)p}(j-j_t+1)^p)^{\frac qp} \\
\lesssim&\sup_{j_t\in\mathbb{Z}}\sum_{j_s\geq j_t}\sum_{j\geq j_s}\sup_{\lambda>0}\lambda^q|\{x:u_{j,j_s}(x)>\lambda\}|^{\frac qp}2^{2(j-j_s)mq}2^{j(\frac
np-1)q}\\&\times2^{(\frac np-\delta)q(j_t-j_s)}(\frac{j-j_t+1}{j-j_s+1})^q2^{q(j-j_s)(2-4m-\frac np)}2^{\delta (j-j_s)q}(j-j_s+1)^q \\
\lesssim&\sup_{j_s\in\mathbb{Z}}\sum_{j\geq j_s}\sup_{\lambda>0}\lambda^q|\{x:u_{j,j_s}(x)>\lambda\}|^{\frac qp}2^{2(j-j_s)mq}2^{j(\frac np-1)q}  .
\end{align*}

For $q>p$, apply H\"older inequality to $j$. For $0<\delta'<n-\delta p$, we have
\begin{align*}
\sup_{j_t\in\mathbb{Z}}\sum_{j'\leq j_t}(A^{p, m',2}_{j',j_t})^q\lesssim&\sup_{j_t\in\mathbb{Z}}\sum_{j'\leq j_t}\sum_{j\geq j'-2}(\sum_{j_t\leq j_s\leq
j}2^{\delta'j}\sup_{\lambda>0}\lambda^p|\{x:u_{j,j_s}(x)>\lambda\}| \\&\times2^{2(j'-j_t)m'p}2^{nj'}2^{-jp}2^{2p(j-j_s)(1-m)}2^{\delta (j-j')p}(j-j_t+1)^p)^{\frac
qp}\\&\times( \sum_{j\geq j'-2}2^{-\delta'j\frac q{q-p}})^{\frac qp-1}\\
\lesssim&\sup_{j_t\in\mathbb{Z}}\sum_{j'\leq j_t}\sum_{j\geq j'-2}\sum_{j_t\leq j_s\leq
j}(2^{\delta'(j-j')}\sup_{\lambda>0}\lambda^p|\{x:u_{j,j_s}(x)>\lambda\}|\\&\times2^{2(j'-j_t)m'p}2^{nj'}2^{-jp}2^{2p(j-j_s)(1-m)}2^{\delta (j-j')p}(j-j_t+1)^{p+\lvert1-\frac
pq\rvert})^{\frac qp}.
\end{align*}
Since $\delta'<n-\delta p$, similarly, we have
$$\sup_{j_t\in\mathbb{Z}}\sum_{j'\leq j_t}(A^{p, m',2}_{j',j_t})^q\lesssim\sup_{j_s\in\mathbb{Z}}\sum_{j\geq
j_s}\sup_{\lambda>0}\lambda^q|\{x:u_{j,j_s}(x)>\lambda\}|^{\frac qp}2^{2(j-j_s)mq}2^{j(\frac np-1)q}.$$



For the third case, just as we proved for the second case, we have also
\begin{align*}
&f^3_{j',j_t}(x)\\\lesssim &   \sup_{2^{-2j_t}\leq t<2^{2-2j_t}} \sum\limits_{(\epsilon',k')\in \Gamma}\sum\limits_{j\geq  j'-2 } \int^t_{\frac t2}
2^{j'-j}2^{2j}\\&\sum\limits_{\epsilon,k}
\frac{2^{\frac n2j}\left\vert u^{\epsilon}_{j,k}(s)\right\vert}
{\left(1+ \left\vert2^{j'-j}k-k'\right\vert\right)^{N} }(s2^{2j})^{-m} ds\chi(2^{j'}x-k')\\
\lesssim& \sup_{2^{-2j_t}\leq t<2^{2-2j_t}} \sum\limits_{j\geq  j'-2 } \sum_{\frac{-\log_2t}{2} \leq j_s< \frac{5-\log_2t}{2}}\int^{2^{2-2j_s}}_{2^{-2j_s}}
2^{j'+j}(s2^{2j})^{-m}dsM(u_{j,j_s})(x)\\
\lesssim&\sum_{j\geq  j'-2 } \sum_{j_t \leq j_s\leq j_t+1}
2^{j+j'-2j_s}2^{2m(j_s-j)}M(u_{j,j_s})(x) .
\end{align*}
For $0<\delta<2m'+\frac np$
\begin{align*}
&\sup_{j_t\in\mathbb{Z}}\sum_{j'\leq j_t}(A^{p, m',3}_{j',j_t})^q\\=&\sup_{j_t\in\mathbb{Z}}\sum_{j'\leq j_t}2^{2(j'-j_t)m'q}2^{j'q(\frac
np-1)}\sup_{\lambda>0}\lambda^q|\{x:f^3_{j',j_t}(x)>\lambda\}|^{\frac qp}\\
\lesssim & \sup_{j_t\in\mathbb{Z}}\sum_{j'\leq j_t}2^{2(j'-j_t)m'q}2^{j'q(\frac np-1)}\sup_{\lambda>0}\lambda^q[\sum_{j\geq j'-2}\sum_{j_t\leq j_s\leq
j_t+1}|\{x:M(u_{j,j_s})(x) 2^{j+j'-2j_s}\\&2^{2(j_s-j)m} >\frac{2^{-\delta j}}{\sum_{j\geq j'-2}2^{-\delta j}}\lambda\}|]^{\frac qp}\\
\lesssim & \sup_{j_t\in\mathbb{Z}}\sum_{j'\leq j_t}[2^{2(j'-j_t)m'p}2^{j'p(\frac np-1)}\sum_{j\geq j'-2}\sum_{j_t\leq j_s\leq
j_t+1}\sup_{\lambda>0}\lambda^p|\{x:M(u_{j,j_s})(x)>\lambda\}| \\&2^{(j+j'-2j_s)p}2^{2p(j_s-j)m}(\frac{\sum_{j\geq j'-2}2^{-\delta j}}{2^{-\delta j}})^p]^{\frac qp}\\
\lesssim &  \sup_{j_t\in\mathbb{Z}}\sum_{j'\leq j_t}[\sum_{j\geq j'-2}\sum_{j_t\leq j_s\leq
j_t+1}\sup_{\lambda>0}\lambda^p|\{x:(Mu_{j})_{j_s}(x)>\lambda\}|2^{2(j'-j_t)m'p}2^{(n-p)j'}\\&\times2^{(j+j'-2j_s)p}2^{2p(j_s-j)m}2^{\delta (j-j')p}]^{\frac qp}\\
\lesssim &  \sup_{j_t\in\mathbb{Z}}\sum_{j'\leq j_t}[\sum_{j\geq j'-2}\sum_{j_t\leq j_s\leq j_t+1}\sup_{\lambda>0}\lambda^p|\{x:u_{j,j_s}(x)>\lambda\}|
2^{2(j'-j_t)m'p}2^{(n-p)j'}\\&\times2^{(j+j'-2j_s)p}2^{2p(j_s-j)m}2^{\delta (j-j')p}]^{\frac qp}\\
\lesssim &  \sup_{j_s\in\mathbb{Z}}\sum_{j'\leq j_s}[\sum_{j\geq j'-2}\sup_{\lambda>0}\lambda^p|\{x:u_{j,j_s}(x)>\lambda\}|
2^{2(j'-j_s)m'p}2^{(n-p)j'}\\&\times2^{(j+j'-2j_s)p}2^{2p(j_s-j)m}2^{\delta (j-j')p}]^{\frac qp}  .
\end{align*}
For $q\leq p$, because $\delta<2m'+\frac np$, since $j\geq j_t$ and $j\geq j_s-1$, we have
 \begin{align*}
&\sup_{j_t\in\mathbb{Z}}\sum_{j'\leq j_t}(A^{p, m',3}_{j',j_t})^q\\
\lesssim&\sup_{j_s\in\mathbb{Z}}\sum_{j\geq j_s-2}\sum_{j'\leq j_s}[\sup_{\lambda>0}\lambda^p|\{x:u_{j,j_s}(x)>\lambda\}|2^{2(j'-j_s)m'p}2^{(n-p)j'}2^{(j+j'-2j_s)p}\\&\times
2^{2p(j_s-j)m}2^{\delta (j-j')p}]^{\frac qp}\\
\lesssim&\sup_{j_s\in\mathbb{Z}}\sum_{j\geq j_s-2}\sup_{\lambda>0}\lambda^q|\{x:u_{j,j_s}(x)>\lambda\}|^{\frac qp}2^{2qm(j-j_s)}2^{j(\frac np-1)q}2^{(n-2p+4mp-\delta
p)(j_s-j)\frac qp}\\
\lesssim&\sup_{j_s\in\mathbb{Z}}\sum_{j\geq j_s}\sup_{\lambda>0}\lambda^q|\{x:u_{j,j_s}(x)>\lambda\}|^{\frac qp}2^{2qm(j-j_s)}2^{j(\frac np-1)q}  .
\end{align*}
For $q>p$, use H\"older inequality to $j$. For $0<\delta'<2m'p+n-\delta p$, we have
\begin{align*}
\sup_{j_t\in\mathbb{Z}}\sum_{j'\leq j_t}(A^{p, m',3}_{j',j_t})^q
\lesssim &  \sup_{j_s\in\mathbb{Z}}\sum_{j'\leq j_s}\sum_{j\geq j'-2}[2^{\delta'(j-j')}\sup_{\lambda>0}\lambda^p|\{x:u_{j,j_s}(x)>\lambda\}|
2^{2(j'-j_s)m'p}\\&\times2^{(n-p)j'}2^{(j+j'-j_s)p}2^{2p(j_s-j)m}2^{\delta (j-j')p}]^{\frac qp}  .
\end{align*}
Since that $\delta'<2m'p+n-\delta p$, we have
\begin{align*}
\sup_{j_t\in\mathbb{Z}}\sum_{j'\leq j_t}(A^{p, m',3}_{j',j_t})^q
\lesssim&\sup_{j_s\in\mathbb{Z}}\sum_{j\geq j_s}\sup_{\lambda>0}\lambda^q|\{x:u_{j,j_s}(x)>\lambda\}|^{\frac qp}2^{2qm(j-j_s)}2^{j(\frac np-1)q}   .
\end{align*}



Then we consider the indices where $t2^{2j'}\geq 1$.
We write the integration $\int^{t}_{0}$ as the sum of three terms:
$\int^{2^{-2j}}_{0}$, $\int^{\frac{t}{2}}_{2^{-2j}}$ and $\int^{t}_{\frac{t}{2}}$.
For the first case, we have
$ |v^{\epsilon}_{j,k}(s)|\leq (s2^{2j})^{-m'} 2^{(1-\frac{n}{2})j}.$
As we did in the case of $t2^{2j'}\leq1$, we have
\begin{align*}
f^1_{j',j_t}(x)=&\sup_{2^{-2j_t}\leq t<2^{2-2j_t}}f^1_{j'}(t,x)=\sup_{2^{-2j_t}\leq t<2^{2-2j_t}}2^{\frac{n}{2}j'} \sum\limits_{(\epsilon',k')\in
\Gamma}|f^{\epsilon',1}_{j',k'}(t)|\chi(2^{j'}x-k')\\
\lesssim &   \sum\limits_{j\geq  j'-2 }\sum\limits_{j_s\geq j+1}e^{-\tilde{c}2^{2(j'-j_t)}} M(u_{j,j_s})(x)2^{j'-j}2^{2(j-j_s)(1-m')}.
\end{align*}
For $s\leq2^{-2j}$, let $A^{p,m,1}_{j',j_t}=2^{2(j-j_t)m}2^{j(\frac np-1)}\sup_{\lambda>0}\lambda |\{x:f^1_{j,j_t}(x)>\lambda\}|^{\frac 1p}$.
That is to say, for $\frac np+4m'-2<\delta<\frac np$, $0<\delta'<2-4m'+\delta-\frac np,$
\begin{align*}
\sup_{j_t\in\mathbb{Z}}\sum_{j'\geq j_t}(A^{p, m,1}_{j',j_t})^q=&\sup_{j_t\in\mathbb{Z}}\sum_{j'\geq j_t}2^{2(j'-j_t)mq}2^{j'q(\frac
np-1)}\sup_{\lambda>0}\lambda^q|\{x:f^1_{j',j_t}(x)>\lambda\}|^{\frac qp}\\
\lesssim &  \sup_{j_t\in\mathbb{Z}}\sum_{j'\geq j_t}[\sum_{j\geq j'-2}\sum_{j_s\geq j+1}\sup_{\lambda>0}\lambda^p|\{x:u_{j,j_s}(x)>\lambda\}|\\&\times
e^{-\tilde{c}p2^{2(j'-j_t)}}2^{2(j'-j_t)mp}2^{nj'}2^{-jp}2^{2p(j-j_s)(1-m')}2^{\delta (j-j')p+\delta'(j_s-j)p}]^{\frac qp}  .
\end{align*}
\par For $q\leq p$,
\begin{align*}
\sup_{j_t\in\mathbb{Z}}\sum_{j'\geq j_t}(A^{p, m,1}_{j',j_t})^q\lesssim&\sup_{j_t\in\mathbb{Z}}\sum_{j'\geq j_t}\sum_{j\geq j'-2}\sum_{j_s\geq
j+1}(\sup_{\lambda>0}\lambda^p|\{x:u_{j,j_s}(x)>\lambda\}|\\&\times e^{-\tilde{c}p2^{2(j'-j_t)}}2^{2(j'-j_t)mp}2^{nj'}2^{-jp}2^{2p(j-j_s)(1-m')}2^{\delta
(j-j')p+\delta'(j_s-j)p})^{\frac qp} \\
\lesssim&\sup_{j_t\in\mathbb{Z}}\sum_{j'\geq j_t}\sum_{j_s\geq j'-1}\sum_{j'-2\leq j\leq j_s-1}(\sup_{\lambda>0}\lambda^p|\{x:u_{j,j_s}(x)>\lambda\}|\\&\times
e^{-\tilde{c}p2^{2(j'-j_t)}}2^{2(j'-j_t)mp}2^{nj'}2^{-jp}2^{2p(j-j_s)(1-m')}2^{\delta (j-j')p+\delta'(j_s-j)p})^{\frac qp}\\
\lesssim&\sup_{j_t\in\mathbb{Z}}\sum_{j_s\geq j_t-1}\sum_{j_t\leq j'\leq j_s+1}\sum_{j'-2\leq j\leq j_s-1}(\sup_{\lambda>0}\lambda^p|\{x:u_{j,j_s}(x)>\lambda\}|\\&\times
e^{-\tilde{c}p2^{2(j'-j_t)}}2^{2(j'-j_t)mp}2^{nj'}2^{-jp}2^{2p(j-j_s)(1-m')}2^{\delta (j-j')p+\delta'(j_s-j)p})^{\frac qp}\\
\lesssim&\sup_{j_t\in\mathbb{Z}}\sum_{j_s\geq j_t-1}\sum_{j_t-2\leq j\leq j_s-1}\sum_{j_t\leq j'\leq j+2}(\sup_{\lambda>0}\lambda^p|\{x:u_{j,j_s}(x)>\lambda\}|\\&\times
e^{-\tilde{c}p2^{2(j'-j_t)}}2^{2(j'-j_t)mp}2^{nj'}2^{-jp}2^{2p(j-j_s)(1-m')}2^{\delta (j-j')p+\delta'(j_s-j)p})^{\frac qp}  .
\end{align*}
Since $2-4m'>\delta'+\frac np-\delta$ and $\delta<\frac np$, we have
\begin{align*}
&\sup_{j_t\in\mathbb{Z}}\sum_{j_s\geq j_t-1}\sum_{j_t-2\leq j\leq j_s-1}\sum_{j_t\leq j'\leq j+2}(\sup_{\lambda>0}\lambda^p|\{x:u_{j,j_s}(x)>\lambda\}|\\&\times
e^{-\tilde{c}p2^{2(j'-j_t)}}2^{2(j'-j_t)mp}2^{nj'}2^{-jp}2^{2p(j-j_s)(1-m')}2^{\delta (j-j')p+\delta'(j_s-j)p})^{\frac qp} \\
\lesssim&\sup_{j_t\in\mathbb{Z}}\sum_{j_s\geq j_t-1}\sum_{j_t-2\leq j\leq j_s-1}(\sup_{\lambda>0}\lambda^p|\{x:u_{j,j_s}(x)>\lambda\}|\\&\times
2^{nj_t}2^{-jp}2^{2p(j-j_s)(1-m')}2^{\delta'(j_s-j)p}2^{\delta p(j-j_t)})^{\frac qp} \\
\lesssim&\sup_{j_t\in\mathbb{Z}}\sum_{j_s\geq j_t-1}\sum_{j\leq
j_s-1}(\sup_{\lambda>0}\lambda^p|\{x:u_{j,j_s}(x)>\lambda\}|2^{2p(j-j_s)m'}2^{j(n-p)}\\&\times2^{(j_t-j_s)(n-\delta p)}2^{(j-j_s)p(2-4m'-\delta'-\frac np+\delta )})^{\frac qp}
\\
\lesssim&\sup_{j_s\in\mathbb{Z}}\sum_{j\leq j_s}\sup_{\lambda>0}\lambda^q|\{x:u_{j,j_s}(x)>\lambda\}|^{\frac qp}2^{2(j-j_s)m'q}2^{j(\frac np-1)q}   .
\end{align*}
For $q>p$, use H\"older inequality to $j$ and $j_s$. For $0<\delta''<n-\delta p$, we have
\begin{align*}
&\sup_{j_t\in\mathbb{Z}}\sum_{j'\leq j_t}(A^{p, m,1}_{j',j_t})^q\\\lesssim&\sup_{j_t\in\mathbb{Z}}\sum_{j'\leq j_t}\sum_{j\geq j'-2}\sum_{j_s\geq
j+1}(2^{\delta''j_s}\sup_{\lambda>0}\lambda^p|\{x:u_{j,j_s}(x)>\lambda\}|\\&\times e^{-\tilde{c}p2^{2(j'-j_t)}}2^{2(j'-j_t)m'p}2^{nj'}2^{-jp}2^{2p(j-j_s)(1-m')}2^{\delta
(j-j')p+\delta'(j_s-j)p})^{\frac qp}\\&\times( \sum_{j\geq j'-2}\sum_{j_s\geq j+1}2^{-\delta''j_s\frac q{q-p}})^{\frac qp-1}\\
\lesssim&\sup_{j_t\in\mathbb{Z}}\sum_{j'\leq j_t}\sum_{j\geq j'-2}\sum_{j_s\geq j+1}(2^{\delta''(j_s-j')}\sup_{\lambda>0}\lambda^p|\{x:u_{j,j_s}(x)>\lambda\}|\\&\times
e^{-\tilde{c}p2^{2(j'-j_t)}}2^{2(j'-j_t)m'p}2^{nj'}2^{-jp}2^{2p(j-j_s)(1-m')}2^{\delta (j-j')p+\delta'(j_s-j)p})^{\frac qp}.
\end{align*}
Because $\delta''<n-\delta p$, similarly, we have
$$\sup_{j_t\in\mathbb{Z}}\sum_{j'\leq j_t}(A^{p, m,1}_{j',j_t})^q\lesssim\sup_{j_s\in\mathbb{Z}}\sum_{j\leq
j_s}\sup_{\lambda>0}\lambda^q|\{x:u_{j,j_s}(x)>\lambda\}|^{\frac qp}2^{2(j-j_s)m'q}2^{j(\frac np-1)q}.$$



For the second case, we have
$ |v^{\epsilon}_{j,k}(s)|\leq (s2^{2j})^{-m} 2^{(1-\frac{n}{2})j}$.
By lemma \ref{le:2.6}, we have
\begin{align*}
f^2_{j',j_t}(x)=&\sup_{2^{-2j_t}\leq t<2^{2-2j_t}}2^{\frac{n}{2}j'} \sum\limits_{(\epsilon',k')\in \Gamma}|f^{\epsilon',2}_{j',k'}(t)|\chi(2^{j'}x-k')\\
\lesssim &   \sum\limits_{j\geq  j'-2 }\sum\limits_{j_t\leq j_s\leq j} e^{-\tilde{c}2^{2(j'-j_t)}}M(u_{j,j_s})(x)2^{j'-j}2^{2(j-j_s)(1-m)} .
\end{align*}
In this case, for $0<\delta<\frac np$
\begin{align*}
&\sup_{j_t\in\mathbb{Z}}\sum_{j'\geq j_t}(A^{p, m,2}_{j',j_t})^q\\=&\sup_{j_t\in\mathbb{Z}}\sum_{j'\geq j_t}2^{2(j'-j_t)mq}2^{j'q(\frac
np-1)}\sup_{\lambda>0}\lambda^q|\{x:f^2_{j',j_t}(x)>\lambda\}|^{\frac qp}\\
\lesssim &  \sup_{j_t\in\mathbb{Z}}\sum_{j'\geq j_t}[\sum_{j\geq j'-2}\sum_{j_t\leq j_s\leq j}\sup_{\lambda>0}\lambda^p|\{x:u_{j,j_s}(x)>\lambda\}|\\&\times
e^{-\tilde{c}p2^{2(j'-j_t)}}2^{2(j'-j_t)mp}2^{nj'}2^{-jp}2^{2p(j-j_s)(1-m)}2^{\delta (j-j')p}(j-j_t+1)^p]^{\frac qp}  .
\end{align*}
\par For $q\leq p$, according to lemma \ref{le:2.4},
\begin{align*}
\sup_{j_t\in\mathbb{Z}}\sum_{j'\geq j_t}(A^{p, m,2}_{j',j_t})^q\lesssim&\sup_{j_t\in\mathbb{Z}}\sum_{j'\geq j_t}\sum_{j\geq j'-2}\sum_{j_t\leq j_s\leq
j}(\sup_{\lambda>0}\lambda^p|\{x:u_{j,j_s}(x)>\lambda\}| e^{-\tilde{c}p2^{2(j'-j_t)}}\\&\times2^{2(j'-j_t)mp}2^{nj'}2^{-jp}2^{2p(j-j_s)(1-m)}2^{\delta
(j-j')p}(j-j_t+1)^p)^{\frac qp} \\
\lesssim&\sup_{j_t\in\mathbb{Z}}\sum_{j_s\geq j_t}\sum_{j'\geq j_t}\sum_{j\geq j_s}(\sup_{\lambda>0}\lambda^p|\{x:u_{j,j_s}(x)>\lambda\}| e^{-\tilde{c}p2^{2(j'-j_t)}}\\&\times
2^{2(j'-j_t)mp}2^{nj'}2^{-jp}2^{2p(j-j_s)(1-m)}2^{\delta (j-j')p}(j-j_t+1)^p)^{\frac qp}  .
\end{align*}
Since $\delta<\frac np$, we have
\begin{align*}
\sup_{j_t\in\mathbb{Z}}\sum_{j'\geq j_t}(A^{p, m,2}_{j',j_t})^q
\lesssim&\sup_{j_t\in\mathbb{Z}}\sum_{j_s\geq j_t}\sum_{j\geq j_s}(\sup_{\lambda>0}\lambda^p|\{x:u_{j,j_s}(x)>\lambda\}|\\&\times 2^{nj_t}2^{-jp}2^{2p(j-j_s)(1-m)}2^{\delta
(j-j_t)p}(j-j_t+1)^p)^{\frac qp} \\
\lesssim&\sup_{j_t\in\mathbb{Z}}\sum_{j_s\geq j_t}\sum_{j\geq j_s}\sup_{\lambda>0}\lambda^q|\{x:u_{j,j_s}(x)>\lambda\}|^{\frac qp}2^{2(j-j_s)mq}2^{j(\frac
np-1)q}\\&\times2^{(\frac np-\delta)q(j_t-j_s)}(\frac{j-j_t+1}{j-j_s+1})^q2^{q(j-j_s)(2-4m-\frac np)}2^{\delta (j-j_s)q}(j-j_s+1)^q \\
\lesssim&\sup_{j_s\in\mathbb{Z}}\sum_{j\geq j_s}\sup_{\lambda>0}\lambda^q|\{x:u_{j,j_s}(x)>\lambda\}|^{\frac qp}2^{2(j-j_s)mq}2^{j(\frac np-1)q}  .
\end{align*}

For $q>p$, we apply H\"older inequality to $j$. For $0<\delta'<n-\delta p$, we have
\begin{align*}
\sup_{j_t\in\mathbb{Z}}\sum_{j'\leq j_t}(A^{p, m,2}_{j',j_t})^q\lesssim&\sup_{j_t\in\mathbb{Z}}\sum_{j'\geq j_t}\sum_{j\geq j'-2}(\sum_{j_t\leq j_s\leq
j}2^{\delta'j}\sup_{\lambda>0}\lambda^p|\{x:u_{j,j_s}(x)>\lambda\}| e^{-\tilde{c}p2^{2(j'-j_t)}}\\&\times2^{2(j'-j_t)m'p}2^{nj'}2^{-jp}2^{2p(j-j_s)(1-m)}2^{\delta
(j-j')p}(j-j_t+1)^p)^{\frac qp}\\&\times( \sum_{j\geq j'-2}2^{-\delta'j\frac q{q-p}})^{\frac qp-1}\\
\lesssim&\sup_{j_t\in\mathbb{Z}}\sum_{j'\geq j_t}\sum_{j\geq j'-2}\sum_{j_t\leq j_s\leq
j}(2^{\delta'(j-j')}\sup_{\lambda>0}\lambda^p|\{x:u_{j,j_s}(x)>\lambda\}|\\&\times2^{2(j'-j_t)m'p}2^{nj'}2^{-jp}2^{2p(j-j_s)(1-m)}2^{\delta (j-j')p}(j-j_t+1)^{p+\lvert1-\frac
pq\rvert})^{\frac qp}.
\end{align*}
For $\delta'<n-\delta p$, similarly, we have
$$\sup_{j_t\in\mathbb{Z}}\sum_{j'\geq j_t}(A^{p, m,2}_{j',j_t})^q\lesssim\sup_{j_s\in\mathbb{Z}}\sum_{j\geq
j_s}\sup_{\lambda>0}\lambda^q|\{x:u_{j,j_s}(x)>\lambda\}|^{\frac qp}2^{2(j-j_s)mq}2^{j(\frac np-1)q}.$$



For the third case, just as we proved the case of $t2^{2j'}\leq1$, we can also get
\begin{align*}
f^3_{j',j_t}(x)
\lesssim\sum\limits_{j\geq  j'-2 } \sum_{j_t \leq j_s\leq j_t+1}2^{j-j'}2^{2m(j_s-j)}M(u_{j,j_s})(x) .
\end{align*}
For $0<\delta<2m-2+\frac np$,
\begin{align*}
&\sup_{j_t\in\mathbb{Z}}\sum_{j'\geq j_t}(A^{p, m,3}_{j',j_t})^q\\=&\sup_{j_t\in\mathbb{Z}}\sum_{j'\geq j_t}2^{2(j'-j_t)mq}2^{j'q(\frac
np-1)}\sup_{\lambda>0}\lambda^q|\{x:f^3_{j',j_t}(x)>\lambda\}|^{\frac qp}\\
\lesssim &  \sup_{j_t\in\mathbb{Z}}\sum_{j'\geq j_t}[\sum_{j\geq j'-2}\sum_{j_t\leq j_s\leq j_t+1}\sup_{\lambda>0}\lambda^p|\{x:u_{j,j_s}(x)>\lambda\}|
2^{2(j'-j_t)mp}2^{(n-p)j'}\\&\times2^{(j-j')p}2^{2p(j_s-j)m}2^{\delta (j-j')p}]^{\frac qp}\\
\lesssim &  \sup_{j_s\in\mathbb{Z}}\sum_{j'\geq j_s}[\sum_{j\geq j'-2}\sup_{\lambda>0}\lambda^p|\{x:u_{j,j_s}(x)>\lambda\}|
2^{2(j'-j_s)mp}2^{(n-p)j'}\\&\times2^{(j-j')p}2^{2p(j_s-j)m}2^{\delta (j-j')p}]^{\frac qp}  .
\end{align*}
For $q\leq p$, because $\delta<2m-2+\frac np$, we have
 \begin{align*}
&\sup_{j_t\in\mathbb{Z}}\sum_{j'\geq j_t}(A^{p, m,3}_{j',j_t})^q\\
\lesssim&\sup_{j_s\in\mathbb{Z}}\sum_{j\geq j_s-2}\sum_{j_s\leq j'\leq j+2}[\sup_{\lambda>0}\lambda^p|\{x:u_{j,j_s}(x)>\lambda\}|\\&\times
2^{2(j'-j_s)mp}2^{(n-p)j'}2^{(j-j')p}2^{2p(j_s-j)m}2^{\delta (j-j')p}]^{\frac qp}\\
\lesssim&\sup_{j_s\in\mathbb{Z}}\sum_{j\geq j_s-2}\sup_{\lambda>0}\lambda^q|\{x:u_{j,j_s}(x)>\lambda\}|^{\frac qp}2^{2qm(j-j_s)}2^{j(\frac np-1)q}2^{2mq(j_s-j)}\\
\lesssim&\sup_{j_s\in\mathbb{Z}}\sum_{j\geq j_s}\sup_{\lambda>0}\lambda^q|\{x:u_{j,j_s}(x)>\lambda\}|^{\frac qp}2^{2qm(j-j_s)}2^{j(\frac np-1)q}   .
\end{align*}
For $q>p$, use H\"older inequality to $j$. For $0<\delta'<2mp-2p+n-\delta p$, we have
\begin{align*}
\sup_{j_t\in\mathbb{Z}}\sum_{j'\geq j_t}(A^{p, m,3}_{j',j_t})^q
\lesssim &  \sup_{j_s\in\mathbb{Z}}\sum_{j'\geq j_s}\sum_{j\geq j'-2}[2^{\delta'(j-j')}\sup_{\lambda>0}\lambda^p|\{x:u_{j,j_s}(x)>\lambda\}|
2^{2(j'-j_s)mp}\\&\times2^{(n-p)j'}2^{(j-j')p}2^{2p(j_s-j)m}2^{\delta (j-j')p}]^{\frac qp}  .
\end{align*}
Since that $\delta'<2mp-2p+n-\delta p$, we have
\begin{align*}
\sup_{j_t\in\mathbb{Z}}\sum_{j'\geq j_t}(A^{p, m,3}_{j',j_t})^q
\lesssim&\sup_{j_s\in\mathbb{Z}}\sum_{j\geq j_s}\sup_{\lambda>0}\lambda^q|\{x:u_{j,j_s}(x)>\lambda\}|^{\frac qp}2^{2qm(j-j_s)}2^{j(\frac np-1)q}   .
\end{align*}




\section{Bilinear operator $B(u,v)$ and proof of Theorem \ref{mthmain}}\label{sec6}

In this section, we establish the boundedness of the bilinear operator $B(u,v)$
and finally obtain the global wellposedness for initial value in Besov-Lorentz spaces.
We consider first the boundedness of bilinear operator $B(u,v)$.
\begin{theorem}\label{th:bilinear}
The bilinear operator
\begin{equation*}
B(u,v)= \int^{t}_{0} e^{(t-s)\Delta}
\mathbb{P}\nabla (u\otimes v) ds
\end{equation*}
is bounded from $({ ^{m'}_{m} \dot{B}}^{\frac{n}{p}-1, q}_{p,\infty})^{n}\times ({ ^{m'}_{m} \dot{B}}^{\frac{n}{p}-1, q}_{p,\infty})^{n}$
to $({ ^{m'}_{m} \dot{B}}^{\frac{n}{p}-1, q}_{p,\infty})^{n}$.
\end{theorem}

\begin{proof}

For $l,l',l''=1,\cdots,n$, define
\begin{equation}\label{61eq}
B_{l}(u,v)= \int^{t}_{0} e^{(t-s)\Delta}\partial_{l}(u(s,x) v(s,x)) ds,
\end{equation}
\begin{equation}\label{62eq}
B_{l,l',l''}(u,v)= \int^{t}_{0} e^{(t-s)\Delta} \partial_{x_l}\partial_{x_{l'}}\partial_{x_{l''}}(-\Delta)^{-1} (u(s,x) v(s,x)) ds.
\end{equation}

We decompose the product of functions $u$ and $v$ into the form of
\begin{equation}\label{eq:decompose}
\begin{array}{rl}
u(t,x)v(t,x)=&\sum_{j\in\mathbb{Z}}Q_juQ_jv+\sum_{0<\lvert j-j^{\prime}\rvert\leq2}Q_juQ_{j^{\prime}}v+\sum_{j\in\mathbb{Z}}P_{j-2}uQ_jv\\&+\sum_{j\in\mathbb{Z}}Q_juP_{j-2}v

\end{array}
\end{equation}
through the projection operator $P_{j}$ and $Q_{j}$ mentioned below Lemma \ref{lem:2.1111}.

We have to prove the boundedness of operators defined in the above equations \eqref{61eq} and \eqref{62eq}: By the similarity between $\sum_{j\in\mathbb{Z}}Q_juQ_jv$ and
$\sum_{0<\lvert j-j^{\prime}\rvert\leq2}Q_juQ_{j^{\prime}}v$,
by the similarity between $\sum_{j\in\mathbb{Z}}P_{j-2}uQ_jv$ and $\sum_{j\in\mathbb{Z}}Q_juP_{j-2}v$,
we only consider $\sum_{j\in\mathbb{Z}}Q_juQ_jv$ and $\sum_{j\in\mathbb{Z}}P_{j-2}uQ_jv$ respectively.
So we just need to prove the boundedness of paraproduct flow and couple flow:
$
P_{l}(u,v),
P_{l,l',l''}(u,v),
C_{l}(u,v)$ and $C_{l,l',l''}(u,v)$.
We have proved Theorems \ref{para} and \ref{cou} in the above two sections.
Therefore we deduce that the operator $B_{l}(u,v)$ and $B_{l,l',l''}(u,v)$ are bounded.

Each component of the operator $B(u,v)$ can be written as the sum of the operators $B_{l}(u,v)$ and $B_{l,l',l''}(u,v)$.
It implies that  the bilinear operator $B(u,v)$
is bounded from $({ ^{m'}_{m} \dot{B}}^{\frac{n}{p}-1, q}_{p,\infty})^{n}\times ({ ^{m'}_{m} \dot{B}}^{\frac{n}{p}-1, q}_{p,\infty})^{n}$
to $({ ^{m'}_{m} \dot{B}}^{\frac{n}{p}-1, q}_{p,\infty})^{n}$.

\end{proof}

{\bf Now we come to prove Theorem \ref{mthmain}.}
\begin{proof}
\par We establish the wellposedness via the following iterative algorithm:
\begin{equation}\label{eq:6.4}
\begin{split}
u^{(0)}(t,x)&=e^{t\Delta}f(x)\,;\\
u^{(\tau+1)}(t,x)=u^{(0)}(t,x)-B(u&^{(\tau)},u^{(\tau)})(t,x),\qquad   for\;\tau=0,1,2,...\,.
\end{split}
\end{equation}
\par  Given  $1<p<\infty, 1\leq q< \infty, 0\leq m'<\frac{1}{2}, m> 1$ or $1<p<\infty, q= \infty, 0< m'<\frac{1}{2}, m> 1$.
For the term at the first line of the above equation \eqref{eq:6.4},
by Theorem \ref{th:B-to-Y} we obtain that if
$f\in \dot{B}^{\frac{n}{p}-1,q}_{p,\infty}$,
then $ e^{t\Delta} f \in { ^{m'}_{m} \dot{B}}^{\frac{n}{p}-1, q}_{p,\infty}.$
Further, by Theorem \ref{th:bilinear},
the bilinear operator $B(u,v)$
is bounded from $({ ^{m'}_{m} \dot{B}}^{\frac{n}{p}-1, q}_{p,\infty})^{n}\times ({ ^{m'}_{m} \dot{B}}^{\frac{n}{p}-1, q}_{p,\infty})^{n}$
to $({ ^{m'}_{m} \dot{B}}^{\frac{n}{p}-1, q}_{p,\infty})^{n}$.

Finally, by Picard's contraction principle, for any small initial data in $(\dot{B}^{\frac{n}{p}-1,q}_{p,\infty})^{n}$,
we can find a unique global solution in $({ ^{m'}_{m} \dot{B}}^{\frac{n}{p}-1, q}_{p,\infty})^{n}$.
\end{proof}





{\bf Acknowledgments}\quad At the Beijing Conference on Harmonic Analysis and Its Applications in 2023, Professor Zhifei Zhang specifically mentioned how to improve the iteration space of Cannone-Meyer-Planchon \cite{CMP}. These ideas of Professor Zhang have a great influence on the arrangement of our article. Further, Professor Zhang gave us a lot of friendly advice on how to change the iteration space and also gave us some references that he cared about. Here we would like to express our sincere thanks to him.

{\bf Conflict of Interest}\quad  The authors declare that they have no conflict of interest.






\end{document}